\documentclass{scrartcl}       % onecolumn (second format)

\usepackage{graphicx}

\usepackage{amsfonts}
\usepackage{amsmath}
\usepackage{graphicx}
\usepackage{amsopn}
\usepackage{amsthm}
\usepackage{mathabx}
\usepackage{color}

\usepackage[normalem]{ulem}
\usepackage{mathtools}
\usepackage{tikz}
\usetikzlibrary{positioning,matrix}
\usepackage{calc}
\newcommand{\blockbmatrix}[3][]{%
	\begin{tikzpicture}
	[
		baseline=-\the\dimexpr\fontdimen22\textfont2\relax,
		inner sep=1pt,
		outer sep=0pt,
		every left delimiter/.style={xshift=.2ex},
		every right delimiter/.style={xshift=-.2ex}
	]
	\matrix
	[
		matrix of math nodes,
		left delimiter={[},
		right delimiter={]},
		nodes={minimum width=1em},
		execute at begin cell={\mathstrut},
		execute at empty cell={\node{0};},
		ampersand replacement=\&
	] (m)
	{
		#2\\
	};

	\foreach \i/\j in {#1}
	{
		\draw[densely dotted,semithick] (m-\i-\i.north west) rectangle (m-\j-\j.south east);
	}
	
	#3
	\end{tikzpicture}%
}
\newtheorem{example}{Example}
\newtheorem{algorithm}[example]{Algorithm}
\newtheorem{proposition}[example]{Proposition}
\newtheorem{definition}[example]{Definition}
\newtheorem{remark}[example]{Remark}
\newtheorem{lemma}[example]{Lemma}
\newtheorem{corollary}[example]{Corollary}
\newtheorem{theorem}[example]{Theorem}

\newcommand{\vertiii}[1]{{\vert\kern-0.25ex\vert\kern-0.25ex\vert #1 
    \vert\kern-0.25ex\vert\kern-0.25ex\vert}}

\begin{document}

\title{Symmetry and Motion Primitives in Model Predictive Control\thanks{Funding by Deutsche Forschungsgemeinschaft (DFG, grant no. WO 2056/4-1, 6-1) and by Mathematisches Forschungsinstitut Oberwolfach is gratefully acknowledged.} 
}

\author{Kathrin Fla{\ss}kamp\footnote{K. Fla{\ss}kamp, 
              Center for Industrial Mathematics,
              University of Bremen, Germany,
              email: kathrin.flasskamp@uni-bremen.de}
         \and
Sina Ober-Bl\"obaum\footnote{S. Ober-Bl\"obaum,
             Department of Engineering Science,
             University of Oxford, United Kingdom,
             email: sina.ober-blobaum@eng.ox.ac.uk}
		 \and
Karl Worthmann\footnote{K. Worthmann,
          Institut f\"{u}r Mathematik,
         Technische Universit\"{a}t Ilmenau, Germany,
        email: karl.worthmann@tu-ilmenau.de}
}

\date{\today}

\maketitle

\begin{abstract}
	Symmetries, e.g.\ rotational and translational invariances for the class of mechanical systems, %
	allow to characterize solution trajectories of nonlinear dynamical systems. %
	Thus, the restriction to symmetry-induced dynamics, e.g.\ by using the concept of motion primitives, may be considered as a quantization of the system. %
	Symmetry exploitation is well-established in both motion planning and control. %
	However, the linkage between the respective techniques to optimal control is not yet fully explored. %
	In this manuscript, we want to lay the foundation for the usage of symmetries in Model Predictive Control (MPC). % 
	To this end, we investigate a mobile robot example in detail where our contribution is twofold: %
	Firstly, we establish asymptotic stability of a desired set point w.r.t.\ the MPC closed loop, %
	which is also demonstrated numerically by using motion primitives applied to the parallel parking scenario. %
	Secondly, if the optimization criterion is not consistent with the symmetry action, %
	we provide guidelines to rigorously derive stability guarantees based on symmetry exploitation.
%%Insert your abstract here. Include keywords, PACS and mathematical
%%subject classification numbers as needed.
%\keywords{Model Predictive Control\and Geometric Control\and Motion Primitives\and Optimal Control\and Symmetry}
%% \PACS{PACS code1 \and PACS code2 \and more}
%% \subclass{MSC code1 \and MSC code2 \and more}
% \subclass{49M37 \and 93D15 \and 34H15}
%% 	Calculus of variations and optimal control; optimization
%% 	49M37   	Methods of nonlinear programming type
%% 	 	 	Ordinary differential equations
%% 	 	 	34H15  	Stabilization
%% 	 	 	Systems theory; control 
%% 	 	 	 	93D15   	Stabilization of systems by feedback
\end{abstract}

\section{Introduction}

In Model Predictive Control (MPC), a sequence of Optimal Control Problems (OCPs) on finite horizons is iteratively solved %
to approximately solve an OCP on an infinite time horizon while continuously taking into account state measurements, see, e.g.\ \cite{GrunePannek2017,RawlingsMayneDiehl2017} for further details. %
It has to be guaranteed though that the resulting feedback law stabilizes the system at a desired set point %
despite the challenging task of solving OCPs for nonlinear dynamical systems in real-time, see, e.g.\ \cite{Hous2011Automatica,ZeilJone11,GiseDoan13,JereGoul14} and \cite{CairKolm18} for techniques to speed up the numerical solution of the OCPs to be solved in each MPC step.
 
The approach pursued in this paper is based on exploiting structural properties of the underlying dynamical system in MPC. %
While, in contrast to linear systems, nonlinear systems cannot be (globally) characterized by evaluating the spectrum at a desired set point, %
it is still possible to identify characteristic properties which describe global system behavior and are useful in motion planning and optimal control~\cite{Frazzoli2001,FrazBull02}. %
An important system property is the existence of symmetries, namely continuous symmetries represented by Lie groups. %
These induce invariances, i.e.\ the system dynamics are invariant w.r.t.\ the corresponding symmetry actions. %
Mechanical systems, such as cars or helicopters, for instance, are typically invariant w.r.t.\ translations or rotations. Consequently, translations or rotations of a trajectory lead to another trajectory of the mechanical system. %
Further, symmetries induce the existence of basic motions, e.g.\ going straight at constant speed or turning with constant rotational velocity in mechanical systems. %
These basic motions will be called trim primitives or trims, for short, see~\cite{FrDaFe05}. %
Trims can be represented very conveniently, even if general solutions of the dynamical systems cannot be computed by hand. %
We will quantize the nonlinear system dynamics by choosing a finite set of basic motions, which will be called motion primitives, to which the system is restricted in order to approximately solve the original OCP. 

Symmetry exploitation is well-established in control of nonlinear dynamical systems (see e.g.~\cite{BuLe04,Bloch,MarsdenRatiu}),
where the notion of symmetry is based on Noether's theorem, stating that a symmetry of a dynamical system induces a first integral, i.e. a quantity that is preserved along the system trajectory.
In \cite{Sussmann}, Sussmann introduced a definiton of symmetry for optimal control problems which allows to identify first integrals, i.e. quantities which are preserved along the state and adjoint trajectories (the so called \emph{biextremals}). This is a useful tool to solve equations of motion for dynamical systems or control problems because finding first integrals can be used to reduce the dimension of the problem which is the main motivation in the aforementioned works.

Classical planning methods only perform geometrical path planning, see e.g.\ \cite{Lav06} for an overview, and do not take the dynamics of the control system into account.
However, already Dubin showed that solutions consisting of arcs of circles and straight lines are optimal w.r.t.\ path length for system dynamics with constrained turning radius, \cite{Dubins1957}.
This has been extended by Reeds and Shepp in \cite{ReedsShepp1990} to explicit solution formulas for shortest paths of system dynamics that allow going forward and backward.
More recently, control methodologies exploiting the multi-body system dynamics~\cite{GariKobi15} or geometric mechanics~\cite{GuptKala15} were proposed.

The exploitation of symmetry-induced motion primitives for planning problems of nonlinear dynamical systems 
has been first proposed by Frazzoli et al.~\cite{Frazzoli2001,FrDaFe05}.
Moreover, in Frazzoli's approach, optimal motion plans are searched for.
Following the idea of quantization (see also \cite{FrazBull02,Ko08,FOK12}), the motion primitives are partly generated by solving optimal control problems for intermediate problems. %
Finding the best motion primitive sequence can then be written as a mixed-integer optimization problem.
Thus, heuristic approaches for globally solving the sequencing problem can be applied, such as sampling-based road-map algorithms \cite{Ko08}. 
While quantizing itself can be seen as reformulating the dynamics as a hybrid system, the approach can also be applied to systems with intrinsic hybrid, i.e.\ mixed discrete-continuous, behavior~\cite{FHO15}. Recently, the idea of motion primitives has also been applied to autonomous driving problems, see \cite{Frazzoli2016} for  motion planning and \cite{PSOB+17,OP18} for a multiobjective MPC framework. The latter work numerically exploits symmetries to reduce the computational effort for generating a library of motion primitives used within an explicit MPC framework. In contrast to this contribution, trim primitives and stability questions regarding MPC are not considered.

Besides their utility in motion planning problems, we have seen that \textit{basic motions} may help in the analysis and the design of MPC schemes as well as the numerical solution of the underlying OCPs: in \cite{Font01} and \cite{GuHu05} concatenations of such basic motions allowed the construction of stabilizing terminal regions and costs for several examples with non-stabilizable linearization (with the peculiarity that the desired set point was not contained in the interior of the terminal region). In particular, the mobile robot was used as a prototype application since its nonholonomic nature "makes the stabilization of this system challenging; see \cite{Asto96}" according to~\cite[p.\ 136]{Font01}. More recently, basic motions were also utilized in the analysis of MPC schemes without terminal costs and constraints, see, e.g.\ \cite{Wort2016CST}. %
However, a formal connection between motion primitives and the proposed techniques has not been established yet. %
The main goal of this work is to lay the foundation for a link between the already quite mature technique of motion primitives and MPC. %
To this end, we revisit the example of the mobile robot in detail in order to directly illustrate the appropriateness of the proposed methodology. %
Here, after introducing the notion of symmetries for optimal control problems based on the definitions for dynamical systems, the contribution of this work is twofold: %
Firstly, for stage cost, which are consistent with the symmetry action, we show recursive feasibility and asymptotic stability of the origin w.r.t.\ the MPC closed loop. %
The key idea is to show that the control effort is uniformly distributed, which also explains why the iterative nature of MPC leads to reduced costs in comparison to the finite horizon optimal control problems. %
Moreover, we prove finite time convergence if motion primitives are used. %
Secondly, we show that the \textit{basic motions} used in~\cite{Wort2015NMPC} were trims and we show that --~besides not using the wording~-- the characteristic properties of trims, namely that a suitable quantization of the system dynamics allows to derive sufficiently good bounds on the value function, was of key importance for the deduced results. This may, e.g., also pave the way for a verification of the distributed controllability assumption introduced in~\cite{GrueneWorthmann2010DMPC}.

The remainder of the paper is organized as follows. In Section~\ref{sec:Preliminaries}, symmetry, motion and trim primitives are introduced for dynamical control systems and illustrated for the example of the mobile robot. %
In Section~\ref{section:symmOCP}, the terminology of symmetries and invariances is extended to optimal control problems. 
Then, in Section~\ref{sec:sym_MPC}, the MPC scheme and a class of admissible control functions are introduced which guarantee the restriction to control trajectories along trim primitives. Convergence of the MPC closed loop trajectory to the origin is shown. In Section~\ref{sec:numerics} we illustrate the MPC approach with motion primitives for the mobile robot.

\section{Symmetries and Motion Primitives}\label{sec:Preliminaries}

Let the system dynamics be given by the ordinary differential equation
\begin{align}\label{NotationSystemDynamics}
		\dot{\mathbf{x}}(t) & = f(\mathbf{x}(t),\mathbf{u}(t))
\end{align}
with initial condition $\mathbf{x}(0) = \mathbf{x}^0$. Here, $\mathbf{x}(t) \in M \subseteq\mathbb{R}^n$ and $\mathbf{u}(t) \in \mathbb{R}^m$ denote the state and the control at time $t \geq 0$ respectively, where $M$ is an $n$-dimensional manifold. Let $\mathcal{T}M$ denote the tangent bundle of $M$. We assume that the map~$f: M \times \mathbb{R}^m \rightarrow \mathcal{T}M$ is continuous and locally Lipschitz w.r.t.\ its first argument in order to guarantee existence and uniqueness of the solution $\varphi_{u}(\cdot;\mathbf{x}^0)$ on its maximal interval of existence $\mathcal{I}_{\mathbf{x}^0,u}$ for $u \in \mathcal{L}^{\infty}_{\operatorname{loc}}([0,\infty),\mathbb{R}^m)$. 
$\mathcal{L}^{\infty}_{\operatorname{loc}}([0,\infty),\mathbb{R}^m)$, $m \in \mathbb{N}$, denotes the space of Lebesgue-measurable and locally absolutely integrable functions.
If we restrict the domain of the control function~$u$ to the compact interval $[0,T]$, $u|_{[0,T]} \in \mathcal{L}^{\infty}([0,T],\mathbb{R}^m)$ and the solution uniquely exists on $\mathcal{I}_{\mathbf{x}^0,u} \cap [0,T]$. 

Throughout this manuscript, we consider the following example to illustrate the definitions, concepts, and results. The robot is modeled by a kinematic model for a $4$-wheeled vehicle which autonomously moves in the 2-dimensional plane under nonholonomic constraints due to the wheels.
\begin{example}[Mobile Robot; $n=3$, $m=2$]\label{ExampleMobileRobot}
	The system dynamics of the mobile robot are governed by
\begin{align}\label{NotationMobileRobot}
	f(\mathbf{x},\mathbf{u}) & = \begin{pmatrix}
		\cos x_3 \\ \sin x_3 \\ 0
	\end{pmatrix}
	u_1 + \begin{pmatrix}
		0 \\ 0 \\ 1
	\end{pmatrix} u_2
\end{align}
where $x_1$ and $x_2$ denote the position of the robot in the plane while $x_3$ represents its orientation and thus, $M = \mathbb{R}^2 \times S^1$. Since $f$ is globally Lipschitz continuous, finite escape times can be excluded such that existence and uniqueness of the solution $\varphi_{u}(t;\mathbf{x}^0)$, $t \in [0,\infty)$, is guaranteed. 
\end{example}

For Example~\ref{ExampleMobileRobot}, a translational invariance (w.r.t.\ the position) can be observed, i.e. 
\begin{equation}\label{eq:RobotInvariance}
	\varphi_u(t;\mathbf{x}^0) + \Delta x = \varphi_u(t;\mathbf{x}^0 + \Delta \mathbf{x}) \qquad\forall\,(t,u) \in \mathbb{R}_{\geq 0} \times \mathcal{L}^{\infty}_{\operatorname{loc}}([0,\infty),\mathbb{R}^2)
\end{equation}
holds for all $\Delta \mathbf{x} = (\Delta x_1,\Delta x_2,0)^\top$ with $\Delta x_1$, $\Delta x_2 \in \mathbb{R}$. This equation states that the translation commutes with the flow, i.e.\ we may first translate the initial state~$\mathbf{x}^0$ by $\Delta \mathbf{x}$ and then compute the flow or first solve the differential equation and then apply the translation~$\Delta \mathbf{x}$. We formalize this \textit{commutativity} in the following definition. To this end, recall that a Lie group is a group $(\mathcal{G},\circ)$, which is also a smooth manifold, for which the group operations $(g,h) \mapsto g \circ h$ and $g \mapsto g^{-1}$ are smooth. If, in addition, a smooth manifold~$M$ is given, we call a map $\Psi: \mathcal{G} \times M \rightarrow M$ a left-action of $\mathcal{G}$ on $M$ if and only if the following properties hold:
	\begin{itemize}
		\item $\Psi(e,\mathbf{x}) = \mathbf{x}$ for all $\mathbf{x} \in M$ where $e$ denotes the neutral element of $(\mathcal{G},\circ)$.
		\item $\Psi(g,\Psi(h,\mathbf{x})) = \Psi(g \circ h,\mathbf{x})$ for all $g, h \in \mathcal{G}$ and $\mathbf{x} \in M$.
	\end{itemize}
	For convenience, we define $\Psi_g:M\rightarrow M$ with  $\Psi_g(x):= \Psi(g,x)$ for $g\in G$ and $x\in M$.
\begin{definition}[Symmetry Group]\label{def:symmetryGroup}
	Let $M$ be a smooth manifold, $(\mathcal{G},\circ)$ a Lie-group, and $\Psi$ a left-action of $\mathcal{G}$ on $M$. Then, we call the triple $(\mathcal{G},M,\Psi)$ a \emph{symmetry group} of the system $\dot{\mathbf{x}}(t) = f(\mathbf{x}(t),\mathbf{u}(t))$ if the property
\begin{align}\label{NotationCommutativityFlowSymmetry}
	\varphi_u(t;\Psi(g,\mathbf{x}^0)) = \Psi(g,\varphi_u(t;\mathbf{x}^0)) \qquad\forall\ (t,g,\mathbf{x}^0) \in \mathbb{R}_{\geq 0} \times \mathcal{G} \times M
\end{align}
holds for all $u \in \mathcal{L}^{\infty}_{\operatorname{loc}}([0,\infty),\mathbb{R}^m)$.
\end{definition}

Next, we compute the symmetry group of the mobile robot introduced in Example~\ref{ExampleMobileRobot} to illustrate Definition~\ref{def:symmetryGroup}. Here, a formal inspection reveals that rotational invariance is combined with translational invariance. In general, the symmetry group of mechanical systems is a subgroup of 
$SE(n) := T(n) \rtimes SO(n)$, where $T(n)$ is the group of translations and $SO(n)$ is the special orthogonal group, which can be represented by the set of matrices $\{R \in \mathbb{R}^{n \times n}: R^\top R = I \text{ and } \det(R) = 1\}$. Since this is a subgroup of the affine group of $n$~dimensions, there are two ways to represent the elements of $SE(n)$:
\begin{itemize}
	\item Either by a pair $(R,\Delta \mathbf{x})$ with $R \in SO(n)$ and $\Delta \mathbf{x} \in \mathbb{R}^n$. Then, the group action can by represented by $\mathbf{x} \mapsto R \mathbf{x} + \Delta \mathbf{x}$.
	\item Or by the single matrix $A \in \mathbb{R}^{(n+1)\times(n+1)}$ via
		\begin{align*}
			\mathbf{x} \mapsto E A \begin{pmatrix}
				\mathbf{x} \\ 1
			\end{pmatrix} \text{ with } \begin{cases}
				A = \begin{pmatrix} R^{\phantom{\top}} & \Delta \mathbf{x} \\ \mathbf{0}^\top & 1 \end{pmatrix} : R \in SO(n),\ \Delta \mathbf{x} \in \mathbb{R}^n \\
			E = \begin{pmatrix}
				I & \mathbf{0}
			\end{pmatrix} \in \mathbb{R}^{n \times (n+1)}
			\end{cases}
		\end{align*}
		where we have first rewritten $\mathbf{x}$ in homogeneous coordinates (see \cite{Murray1994}) before multiplying by $A$. The projection on the state space $\mathbb{R}^n$ can then by represented by the matrix~$E$.
\end{itemize}

In addition to the translational invariance observed in Equation~\eqref{eq:RobotInvariance}, the \emph{planar} mobile robot also has rotational symmetry as specified in the following proposition.

\begin{proposition}\label{PropositionSymmetryGroup}
	For given $u \in \mathcal{L}^{\infty}_{\operatorname{loc}}([0,\infty),\mathbb{R}^2)$, we consider the flow~$\varphi_u$ generated by the system dynamics~\eqref{NotationMobileRobot}. Then,
\begin{equation*}
	\mathcal{G} := \left\{ \blockbmatrix[1/3,4/4]{
		\cos(\Delta x_3) \& - \sin(\Delta x_3) \& 0 \& \Delta x_1 \\
		\sin(\Delta x_3) \& \phantom{-}\cos(\Delta x_3) \& 0 \& \Delta x_2 \\
		0 \& 0 \& 1 \& \Delta x_3 \\
		0 \& 0 \& 0 \& 1
	}{} : \Delta \mathbf{x} \in \mathbb{R}^2\times S^1 \right\},
\end{equation*}
is a symmetry group of the mobile robot with the matrix multiplication as group action, i.e.~$\Psi(g,\mathbf{x})=E g \begin{pmatrix}
				\mathbf{x} \\ 1
			\end{pmatrix}, g\in \mathcal{G}$, (with $\mathbf{x} \in M \subset \mathbb{R}^3$ represented in homogeneous coordinates).
\end{proposition}
\begin{proof}
	As a subgroup of $SE(2) \times S^1$, $\mathcal{G}$ is a Lie group. The flow of the mobile robot is given by
	\begin{align}
		\varphi_u(t;\mathbf{x}^0) & = \mathbf{x}^0 + \int_0^t \begin{pmatrix}
			\cos( x_0^3 + \int_0^s u_2(\tau)\,\mathrm{d}\tau ) u_1(s) \\
			\sin( x_0^3 + \int_0^s u_2(\tau)\,\mathrm{d}\tau ) u_1(s) \\
			u_2(s)
		\end{pmatrix}\,\mathrm{d}s.
	\end{align}
	Then, direct calculations, using the angle sum formula for sine and cosine, show the Identity~\eqref{NotationCommutativityFlowSymmetry} since both terms can be written as
	\begin{align}\label{NotationFlowAnalyticSolution}
			R_{\Delta x_3} \mathbf{x}^0 + \Delta \mathbf{x} + \int_0^t \begin{pmatrix}
				\cos(x_3^0 + \int_0^s u_2(\tau)\,\mathrm{d}\tau + \Delta x_3)  u_1(s) \\
				\sin(x_3^0 + \int_0^s u_2(\tau)\,\mathrm{d}\tau + \Delta x_3)  u_1(s) \\
				u_2(s)
			\end{pmatrix}\,\mathrm{d}s
	\end{align}
	where the rotation matrix is denoted by 
	$R_{\Delta x_3}$. 
\end{proof}

If Property~\eqref{NotationCommutativityFlowSymmetry} holds, the flow of the system is said to be equivariant w.r.t.\ the symmetry action $\Psi$.
As a consequence, given a trajectory $\varphi_u(\cdot,\mathbf{x}^0)$, new trajectories $\Psi(g,\varphi_u(\cdot,\mathbf{x}^0))$, $g \in \mathcal{G}$, can be generated using~$\Psi$. This family of trajectories --~parametrized in the group element~$g$~-- forms an equivalence class.
\begin{definition}[Motion Primitive]\label{DefinitionMotionPrimitive}
	Let $(\mathcal{G},M,\Psi)$ be a symmetry group in the sense of Definition~\ref{def:symmetryGroup}. Then, two trajectories $\varphi_u(\cdot;\mathbf{x}^0)$ and $\varphi_{u}(\cdot;\bar{\mathbf{x}}^0)$ are called \emph{equivalent}, if there exists $g \in \mathcal{G}$ such that
	\[
		\varphi_u(t;\mathbf{x}^0) = \Psi(g,\varphi_{u}(t;\bar{\mathbf{x}}^0)) \qquad\forall\,t \geq 0.
	\]
	A \emph{motion primitive} is the equivalence class of all trajectories equivalent to  $\varphi_u(\cdot;\mathbf{x}^0)$ w.r.t.\ the left action~$\Psi$.
\end{definition}

Note that $u$ is an arbitrary but fixed control function in Definition~\ref{DefinitionMotionPrimitive}, i.e.\ $u$ is identical for all members of the same  
motion primitive. By slight abuse of notation, we will use the term motion primitive also for a representative of the equivalence class. 
The symmetry action of the mobile robot is illustrated in Figure~\ref{FigureMobileRobotAffineTransformations}.\\

\begin{figure}
	\includegraphics[width=0.85\textwidth]{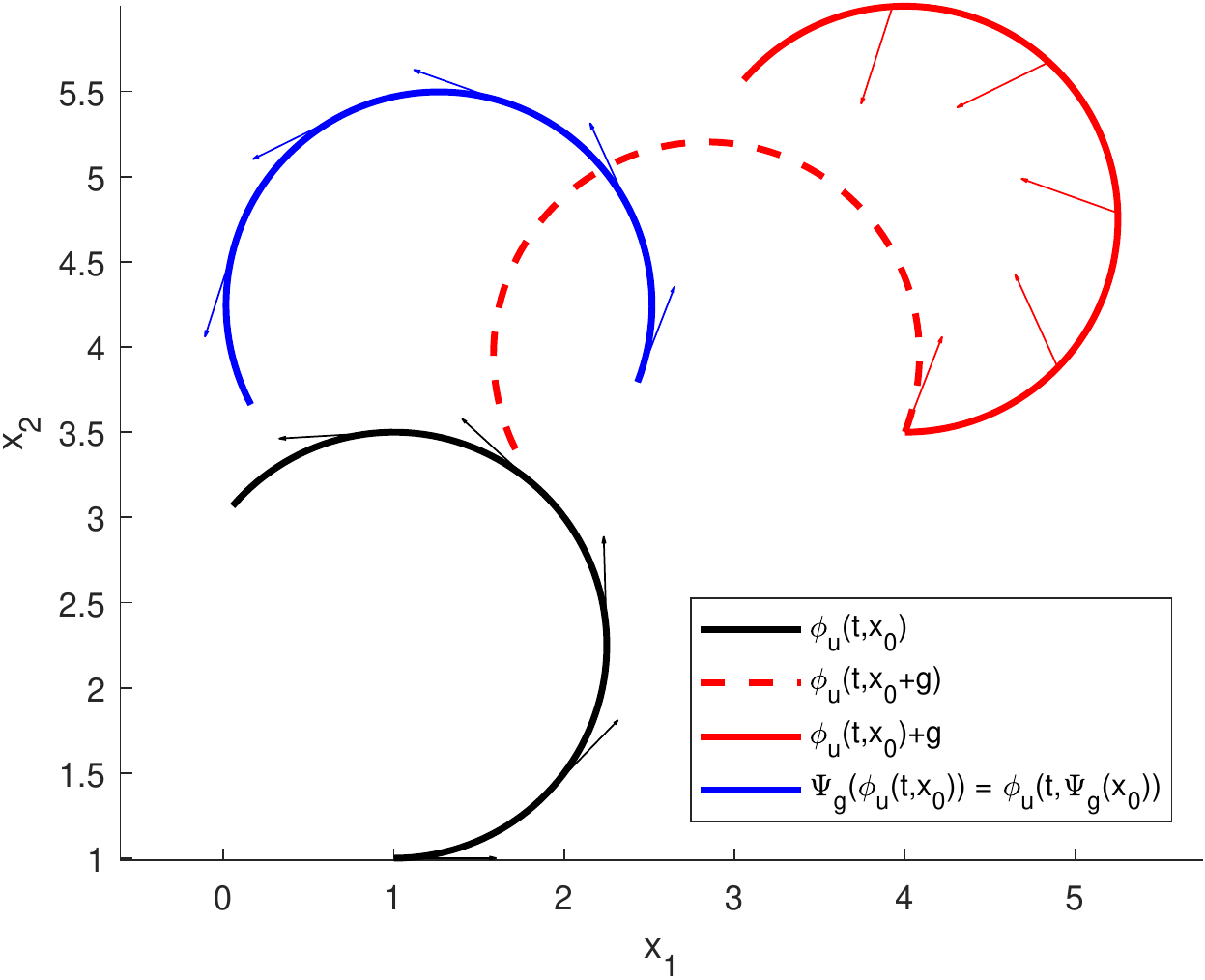}
	\caption{For $u \equiv (u_1\ u_2)^\top$, $u_2 \neq 0$, the projection of the trajectory $\bigcup_{t \in [0,\delta]} \varphi_{u}(t;\mathbf{x}^0)$ is a segment of a circle. Small arrows indicate the current values of angle $x_3$. To illustrate the symmetry shift, consider $\Delta \mathbf{x} = (3,\, 2.5,\, 1.2)^\top$ and the black curve as a starting point.
	Naively, one may think that the symmetry shift simply equals an addition of $g$ to the state. The two red curves illustrate that this is wrong: the dashed red curve is generated by the flow starting at $\mathbf{x}^0 + \Delta \mathbf{x}$ and it does not overlap with the solid red curve, which is a point-wise translation of the black curve by $g$. In fact, the solid red curve cannot even be a system trajectory, since the heading angle $x_3$ of the robot is not tangential to its motion.
	However, the symmetry action of Proposition~\ref {PropositionSymmetryGroup} can be successfully validated: the blue curve is both, the flow of $\varphi_{u}(\cdot,\Psi_g(\mathbf{x}^0))$ as well as the symmetry action applied point-wise to the original (black) solution.}
\label{FigureMobileRobotAffineTransformations}
\end{figure}

Considering constant control functions $\bar{u}$, we may link trajectories $\varphi_{\bar{u}}(t;\mathbf{x}^0)$, $t \geq 0$, to (the positive part of) a corresponding one-parameter subgroup $g_t: \mathbb{R} \to \mathcal{G}$. %
Trajectories $\varphi_{\bar{u}}(\cdot;\mathbf{x}^0)$, which exhibit this link are called trim primitives, trims in short, and can be considered as basic motions. To be more precise, a particular element~$\xi$ of the Lie algebra~$\mathfrak{g}$ is scaled by time $t$ and then mapped via the exponential function to the Lie group~$\mathcal{G}$ to generate the corresponding one-parameter subgroup. We show this connection for the symmetry group of the mobile robot in combination with constant control functions.

\begin{definition}[Trim Primitive] \label{def:Trims}
Let $(\mathcal{G},M,\Psi)$ be a symmetry group in the sense of 
	Definition~\ref{def:symmetryGroup}.
	Then, a trajectory $\varphi_u(\cdot;\mathbf{x}^0)$ is called a \emph{trim primitive} if there exists a Lie algebra element $\xi \in\mathfrak{g}$ such that
	\[ 
		\varphi_u(t;\mathbf{x}^0) = \Psi(\exp(\xi t),\mathbf{x}^0) \quad\text{ and }\quad u(t) \equiv \bar{u} = \text{const. } \quad  \forall t \geq 0.\]
\end{definition}

\begin{proposition}\label{prob:RobotTrims}
	Let $(\mathcal{G},M,\Psi)$ be the symmetry group of Proposition~\ref{PropositionSymmetryGroup} and $\mathfrak{g}$ denote the corresponding Lie algebra. Then, for the constant control function $u:[0,\delta] \rightarrow \mathbb{R}^2$, $\delta > 0$, defined by $u(t) = (u_1\ u_2)^\top$ for all $t \in [0,\delta]$ and the initial value~$\mathbf{x}^0$, we get
	\begin{align}\label{NotationTrimComputation}
		\varphi_u(t;\mathbf{x}^0) = \Psi_{g_t}(\mathbf{x}^0) := \Psi(g_t,\mathbf{x}^0)
	\end{align}
	for $g_t = 	\exp(\xi t)$ with $\xi = (v_1\ v_2\ u_2)^\top \in \mathfrak{g}$ defined by
		\begin{align*}
		v_1 & := u_1 \cos(x^0_3) + u_2 x_2^0, \\
		v_2 & := u_1 \sin(x^0_3) - u_2 x_1^0.
	\end{align*}
	In particular, $\Psi_{g_t}(\mathbf{x}^0) = R_{u_2t} \mathbf{x}^0 + \mathbf{b}_{g_t}$ with the translation vector
	\begin{align*}
		\mathbf{b}_{g_t} := \begin{pmatrix}
			u_2^{-1} \left(v_1 \sin(u_2 t) - v_2 (1-\cos(u_2 t)) \right) \\
			u_2^{-1} \left(v_1 (1-\cos(u_2 t)) + v_2 \sin(u_2 t) \right) \\
			u_2 t
		\end{pmatrix} \quad\text{ and }\quad \mathbf{b}_{g_t} := \begin{pmatrix}
			v_1t \\ v_2t \\ 0
		\end{pmatrix}
	\end{align*}
	for $u_2 \neq 0$ and $u_2 = 0$, respectively.
\end{proposition}
\begin{proof}
	The Lie algebra for a Lie group that consists of rotation matrices is given by skew symmetric matrices \cite{Murray1994}. The corresponding Lie algebra is one-dimensional and can be represented as
	\begin{align*}
		\begin{pmatrix}
			0 & -u_2 & 0 \\
			u_2 & 0 & 0  \\
			0 & 0 & 0
		\end{pmatrix}.
	\end{align*}
The Lie algebra that corresponds to translations in $\mathbb{R}^n$ and $S^1$ is simply $\mathbb{R}^n$ and $\mathbb{R}$, respectively \cite{Murray1994}. Together, we obtain that every element~$\xi$ of the Lie algebra can be represented by the %skew-symmetric
	matrix
	\begin{align*}
		\begin{pmatrix}
			0 & -u_2 & 0 & v_1 \\
			u_2 & 0 & 0 & v_2 \\
			0 & 0 & 0 & u_2 \\
			0 & 0 & 0 & 0
		\end{pmatrix}
	\end{align*}
	using the triple $(v_1\ v_2\ u_2)^\top$. Then, since we are using the homogeneous representation, the linear scaling $t \xi$ by a factor $t \in \mathbb{R}$ and the exponential $\exp: \mathfrak{g} \rightarrow \mathcal{G}$ can by directly calculated by first scaling the representation matrix and, then, computing the matrix exponential ($u$ is constant). Doing so yields
	\begin{align*}
		\exp(\xi t) = \begin{pmatrix}
			\cos(u_2 t) & -\sin(u_2 t) & 0 & \frac 1 {u_2} (v_1 \sin(u_2 t) - v_2 (1-\cos(u_2 t))) \\
			\sin(u_2 t) & \phantom{-}\cos(u_2 t) & 0 & \frac 1 {u_2} (v_1 (1-\cos(u_2 t) - v_2 \sin(u_2 t))) \\
			0 & 0 & 1 & u_2 t \\
			0 & 0 & 0 & 1
		\end{pmatrix},
	\end{align*}
	which shows the claimed representation of $\Psi_{g_t}(\mathbf{x}^0)$. Moreover, the left hand side of Equation~\eqref{NotationTrimComputation} equals
	\begin{align*}
			\mathbf{x}^0 + \begin{pmatrix}
				\frac {u_1}{u_2} (\sin(x_3^0) - \sin(x_3^0 + t u_2)) \\
				\frac {u_1}{u_2} (\cos(x_3^0 + t u_2) - \cos(x_3^0)) \\
				t u_2(s)
			\end{pmatrix}
	\end{align*}
	by solving the integrals in the representation~\eqref{NotationFlowAnalyticSolution}. Then, direct calculations using the angle sum formula for sine and cosine show
Equation~\ref{NotationTrimComputation}. 
\end{proof}

Proposition~\ref{prob:RobotTrims} shows that the mobile robot exhibits trims according to Definition~\ref{def:Trims} for all constant control functions and arbitrary initial values. If the rotational speed~$u_2(t)$ is constant and not equal to zero, trims result in a circular motion (otherwise it is a straight line), see e.g.\ the black curve in Figure~\ref{FigureMobileRobotAffineTransformations}.
In general, trims are attractive since particular solutions of a nonlinear system can be computed while no analytic expression for the general solution is available.

\begin{remark}[Dynamical Systems \& Trims] 
	In dynamical systems without control input, motions generated by symmetry actions are called \emph{relative equilibria} since these motions have to be constant in coordinates which are non-symmetric \cite{MarsdenRatiu,Bloch}. Constructive approaches to find trims can be deduced from symmetry reduction methods.
For mechanical systems with cyclic coordinates, this has been worked out in \cite{FOK12}.\\
Note that any solution with constant control generates a trim for Example~\ref{ExampleMobileRobot} since it is a kinematic model (and, thus, the velocities are directly controlled). If actuator dynamics are added, it becomes a second order mechanical system. For this system class, trims typically correspond to motions with constant velocities in body-fixed frame. 
\end{remark}

\begin{remark}[Alternative Characterization of Symmetry Groups]
	Let $f(\cdot,\mathbf{u})$ maps from a smooth manifold~$M$ to the tangential bundle $\mathcal{T}M$ %space $\mathcal{T}_{\mathbf{x}}M$ 
	of $M$. The symmetry action as a map $\Psi: \mathcal{G} \times M \rightarrow M$ can be lifted to $\mathcal{T}_\mathbf{x}M$ for $\mathbf{x}\in M$ %$\mathcal{T}_{\mathbf{x}} M$ 
	via
	\begin{align*}
		\Psi^{\mathcal{T}_\mathbf{x}M}: \mathcal{G} \times \mathcal{T}_xM \rightarrow \mathcal{T}_\mathbf{x}M,\quad \Psi^{\mathcal{T}_\mathbf{x}M}\left(g, \mathbf{v} \right) = \frac {\mathrm{d} \Psi_g}{\mathrm{d}\mathbf{x}} (\mathbf{x}) \cdot \mathbf{v}.
	\end{align*}
	Then, the vector field is said to be equivariant w.r.t.\ the symmetry action~$\Psi$ if $f(\Psi_g(\mathbf{x}),\mathbf{u})   = \Psi_g^{\mathcal{T}_xM} (f(\mathbf{x},\mathbf{u}))\ \forall \mathbf{x}\in M$. A direct calculation shows the equivalence of this condition to Property~\eqref{NotationCommutativityFlowSymmetry}: Application of Equation~\eqref{NotationCommutativityFlowSymmetry} and its time derivative yields
\begin{align*}
		f(\Psi_g(\mathbf{x}),\mathbf{u}) & {=}  f(\Psi_g(\varphi_u(t;\mathbf{x^0})) ,\mathbf{u})\stackrel{\eqref{NotationCommutativityFlowSymmetry}}{=}f(\varphi_u (t;\Psi_g(\mathbf{x^0})),\mathbf{u})= \frac {\mathrm{d}\varphi_u}{\mathrm{d}t} (t;\Psi_g(\mathbf{x^0})) \\ &\stackrel{\eqref{NotationCommutativityFlowSymmetry}}{=} \frac {\mathrm{d}\Psi_g}{\mathrm{d}x} (\varphi_u(t;\mathbf{x^0})) \frac {\mathrm{d}\varphi_u}{\mathrm{d}t} (t;\mathbf{x^0})  = \Psi_g^{\mathcal{T}_xM}(f(\mathbf{x},\mathbf{u})).
	\end{align*}
\end{remark}

\begin{remark}[Alternative proof of Proposition~\ref{PropositionSymmetryGroup}]
Instead of showing the invariance of the flow $\varphi_u$ (Property \eqref{NotationCommutativityFlowSymmetry}) we can, alternatively, show the equivariance of the vector field $f$ to prove Proposition~\ref{PropositionSymmetryGroup}. With $\Psi_g(\mathbf{x})= R \mathbf{x} + \Delta \mathbf{x}$ the lifted action $\Psi_g^{\mathcal{T}_xM}$ is given by $\frac {\mathrm{d} \Psi_g}{\mathrm{d}\mathbf{x}} (\mathbf{x}) = R$ and thus, it follows with the vector field $f$ given in \eqref{NotationMobileRobot}
\begin{align*}
	f(\Psi_g(\mathbf{x}),\mathbf{u}) & = \begin{pmatrix}
		\cos (x_3 + \Delta x_3) \\ \sin (x_3 + \Delta x_3) \\ 0
	\end{pmatrix}
	u_1 + \begin{pmatrix}
		0 \\ 0 \\ 1
	\end{pmatrix} u_2 \\
	&= \begin{pmatrix}
		\cos (\Delta x_3) & -\sin (\Delta x_3) & 0\\ \sin (\Delta x_3)&\cos (\Delta x_3) & 0 \\ 0& 0& 1
	\end{pmatrix} \begin{pmatrix}
		\cos (x_3) u_1 \\ \sin (x_3 ) u_1 \\ u_2 \end{pmatrix} = 
		   \Psi_g^{\mathcal{T}_xM}(f(\mathbf{x},\mathbf{u})).
		%R\ f(\mathbf{x},\mathbf{u})
\end{align*}
\end{remark}

\section{Symmetry and Optimal Control}\label{section:symmOCP}

For given optimization horizon~$T$, $T \in \mathbb{R}_{>0}$, we consider the Optimal Control Problem (OCP)
\begin{equation}
	\boxed{\begin{aligned}
		\text{Minimize} \quad & \int_0^T \ell (\mathbf{x}(t),\mathbf{u}(t))\,\mathrm{d}t \quad\text{w.r.t.\ } u \in \mathcal{L}^{\infty}([0,T],\mathbb{R}^m), x \in \mathcal{AC}([0,T],M)\\
		\text{subject to} \quad & \ r(\mathbf{x}(0),\mathbf{x}(T)) = 0 & \hspace*{-5cm} \text{(boundary condition)}\\
		&\ g(\mathbf{x}(t),\mathbf{u}(t)) \leq 0,\ t \in [0,T], & \hspace*{-5cm} \text{(state \& control constraint)} \\
		&\ \dot{\mathbf{x}}(t) = f(\mathbf{x}(t),\mathbf{u}(t)),\ t \in [0,T] & \hspace*{-5cm} \text{(system dynamics)}
	\end{aligned}}%
	\tag{OCP}\label{OCP}
\end{equation}
with continuous stage cost $\ell: \mathbb{R}^n \times \mathbb{R}^m \to \mathbb{R}$ and functions $r:\mathbb{R}^n \times \mathbb{R}^n \to \mathbb{R}^{2n}$ and $g:\mathbb{R}^n \times \mathbb{R}^m \to \mathbb{R}^q$, $q \in \mathbb{N}_0$.
 Moreover, $\mathcal{AC}([0,T],M)$, is the space of absolutely continuous functions on the manifold $M$. 

Recall that we identified symmetries of the system dynamics in Section~\ref{sec:Preliminaries}. %
Now, our aim is to take advantage of these symmetries in optimal control. %
Therefore, we are interested in functions~$\ell$, $r$, and $g$ that share the invariance properties w.r.t.\ a symmetry of the system dynamics, see Subsection~\ref{SubsectionConsistentOCP}. Then, in Subsection~\ref{SubsectionInconsistentOCP}, we are concerned with limitations and possible remedies, which might occur if a constraint function or the stage cost do not share the invariance. Finally, in Subsection~\ref{SubsectionMobileRobotOCP}, we formulate the OCP for the example of the mobile robot. Here, we define a set of admissible control functions and show existence of an optimal control.

\subsection{OCPs Consistent with the Invariance of the System Dynamics}\label{SubsectionConsistentOCP}

In the following definition, we precisely state what we mean by saying that the constraints and the stage cost share the invariance of the system dynamics.
\begin{definition}\label{DefinitionInvariantCostsConstraints}
	Consider the system dynamics~\eqref{NotationSystemDynamics}, i.e.\ $\dot{\mathbf{x}}(t) = f(\mathbf{x}(t),\mathbf{u}(t))$, with symmetry group $(\mathcal{G},M,\Psi)$. %
	We call a function \emph{invariant w.r.t.\ the symmetry}, if every state can be replaced by its image under the symmetry action~$\Psi_g$ without changing its value - independent of the particular choice of~$g \in \mathcal{G}$ and $\mathbf{u} \in \mathbb{R}^m$, i.e.
	\begin{align*}
		\ell(\Psi_{g}(\mathbf{x}),\mathbf{u}) & = \ell(\mathbf{x},\mathbf{u}) \qquad \forall\, (g,\mathbf{x},\mathbf{u}) \in \mathcal{G} \times M \times \mathbb{R}^m, \\
		g(\Psi_{g}(\mathbf{x}),\mathbf{u}) & = g(\mathbf{x},\mathbf{u}) \qquad \forall\, (g,\mathbf{x},\mathbf{u}) \in \mathcal{G} \times M \times \mathbb{R}^m, \\
		r(\Psi_{g}(\mathbf{x}),\Psi_g({\bar{\mathbf{x}}})) & = r(\mathbf{x},\bar{\mathbf{x}}) \qquad \forall\, (g,\mathbf{x},\bar{\mathbf{x}}) \in \mathcal{G} \times M \times M.
	\end{align*}
\end{definition}

If the stage cost~$\ell$ and the constraint functions $r,g$ are invariant, all equivalent trajectories, i.e.\ each motion primitive, have the same costs and remain feasible, which directly follows from Definition~\ref{DefinitionInvariantCostsConstraints}. Then, we call~\eqref{OCP} consistent (with the invariance property of the system dynamics), which is motivated by the following proposition. As a direct consequence, also optimality is preserved if \eqref{OCP} is consistent. %
\begin{proposition}\label{PropositionInvarianceOCP}
	Consider \eqref{OCP} and let $(\mathcal{G},M,\Psi)$ be a symmetry group of the system dynamics. Furthermore, let $u \in \mathcal{L}^{\infty}([0,T],\mathbb{R}^m)$ and $x \in \mathcal{AC}([0,T],M)$ satisfy the constraints. Then, if the stage cost~$\ell$ and the constraint functions $r,g$ are invariant, all pairs $(u,\Psi_g(x(\cdot)))$, $g \in \mathcal{G}$, also satisfy the constraints and yield the same costs, i.e.
	\begin{align*}
		\int_{0}^{T} \ell(\mathbf{x}(t),\mathbf{u}(t)) \, \mathrm{d}t = \int_{0}^{T} \ell(\Psi_g(\mathbf{x}(t)),\mathbf{u}(t)) \, \mathrm{d}t.
	\end{align*}
\end{proposition}

Next, we show that the cost function is particularly simple to evaluate along trim primitives, which also explains why $\ell(\mathbf{x}(0),\bar{u})$ is sometimes called \textit{unit cost} of a trim, see \cite{FrDaFe05}.
\begin{proposition}
	Consider a continuous stage cost $\ell: \mathbb{R}^n \times \mathbb{R}^m \to \mathbb{R}$ and let $(\mathcal{G},M,\Psi)$ be a symmetry group of the system dynamics~\eqref{NotationSystemDynamics}.
Further, let a constant control function $u \equiv \bar{\mathbf{u}}$ and a corresponding element $\xi \in \mathfrak{g}$ of the generating Lie algebra $\mathfrak{g}$ be given.
	Then, for given $T>0$, for each
	state trajectory~$x \in \mathcal{AC}([0,T],M)$ such that $\mathbf{x}(t):= \varphi_u(t;\mathbf{x}(0)) = \Psi_{\exp(\xi t)}(\mathbf{x}(0))$ holds, i.e.\ for each trim primitive, the invariant stage cost~$\ell$ satisfies
	\[
		\int_0^T \ell(\mathbf{x}(t),\mathbf{u}(t))\,\mathrm{d}t = T \cdot \ell(\mathbf{x}(0),\bar{\mathbf{u}}).
	\]
\end{proposition}
\begin{proof} We have
	\begin{align*}
		\int_0^T \ell(\mathbf{x}(t),\mathbf{u}(t))\,\mathrm{d}t %= \int_0^T \ell(\mathbf{x}(t),\bar{\mathbf{u}})\,\mathrm{d}t 
		= \int_{0}^{T} \ell(\Psi_{\exp(\xi t)}(\mathbf{x}(0)),\bar{\mathbf{u}}) \, \mathrm{d}t = T\cdot \ell(\mathbf{x}(0),\bar{\mathbf{u}}).
	\end{align*}
	Here, the first equality follows from the trim definition and $u \equiv \bar{\mathbf{u}}$ and the second equality from the invariance of $\ell$. 
\end{proof}
Possible choices for invariant stage costs are the following. Note that also a weighted sum leads to an invariant cost function.
\begin{itemize}
	\item \emph{Minimal control effort}: $\ell(\mathbf{x},\mathbf{u}) = \Vert \mathbf{u} \Vert_{R}^{2}$ with $R$ being some symmetric positive definite matrix.
	\item \emph{Minimum path length}
	\item \emph{Minimum ``fuel consumption''}, i.e.\ $\ell(\mathbf{x},\mathbf{u}) = \Vert \mathbf{u} \Vert$. %
	For some systems, e.g.\ when $u$ models the fuel, it might be desirable to minimize the $L_{1}$ norm of $u$ instead of the $L_{2}$ norm, or a combination of both.
	\item \textit{Minimal time}, i.e.\ $\ell(\mathbf{x},\mathbf{u}) = 1$. %
	Here, we get $\int_{0}^{T} 1\, \mathrm{d}t = T$, i.e.\ the final time~$T$ is free. Then, $T$ is an additional real-valued optimization variable.
\end{itemize}

\subsection{Pitfalls, Inconsistency, and Remedies}\label{SubsectionInconsistentOCP}

In Subsection~\ref{SubsectionConsistentOCP}, we have seen a variety of invariant stage costs. A typical representative for invariant constraints would be a (geometrical) path bridging a certain distance, which explains why motion primitives are often employed for path planning objectives, see, e.g.\ \cite{FrDaFe05}. However, initial value problems with a fixed desired terminal state or quadratic stage costs, which are typically used for stabilization task, lead to inconsistent OCPs as shown in the following. Moreover, we explicate a \textit{modified}/shifted OCP, which allows to recover consistency of the OCP if desired.

Let us start with a (classical) initial condition, i.e.\  $\mathbf{x}(0) = \mathbf{x}^0$ or, equivalently,
\begin{align*}
	 r_i(\mathbf{x}(0),\mathbf{x}(T)) := \mathbf{x}_i(0) - \mathbf{x}^0_i = 0 \qquad\forall\, i \in \{1,\ldots,n\}.
\end{align*}
Note that only the first $n$ components of the function~$r$ describing the boundary conditions are used in this example. The remaining $n$ components would be typically employed to 
enforce meeting the terminal constraint. Plugging in $\Psi_g(\mathbf{x}(0))$ (and also $\Psi_g(\mathbf{x}(T))$ for completeness) yields the condition $\Psi_g(\mathbf{x}_i(0)) - \mathbf{x}^0_i = 0$, which is for $g \in \mathcal{G} \setminus \{e\}$, in general, not satisfied. However, a
 \textit{shifted} OCP with initial condition $\Psi_g(\mathbf{x}^0)$ --~a shift with the particular group element~$g$, $g \in \mathcal{G}$~-- is invariant, meaning that feasibility (and optimality) are preserved for the OCP with \textit{shifted} constraints. This can be explained using the representation of the symmetry action in homogeneous coordinates: the translation vector cancels out while the distance is preserved under the respective orthogonal transformation. Similar arguments apply to terminal conditions and quadratic stage costs exemplarily defined as
\begin{align}
	\ell(\mathbf{x},\mathbf{u}) := (\mathbf{x}-\mathbf{x}^\star)^\top Q (\mathbf{x}-\mathbf{x}^\star) + (\mathbf{u}-\mathbf{u}^\star)^\top R (\mathbf{u}-\mathbf{u}^\star)
\end{align}
with symmetric matrices $Q \in \mathbb{R}^{n \times n}$ and $R \in \mathbb{R}^{m \times m}$ (with $Q \succcurlyeq 0$ and $R \succ 0$) and references $\mathbf{x}^\star$ and $\mathbf{u}^\star$ for state and control respectively.

\begin{example}[Shifted OCP]\label{ExampleMobileRobotModifiedOCP}
	The robot of Example~\ref{ExampleMobileRobot} shall be controlled into the final state  $\mathbf{x}^\star = (0,\,0,\,0)^\top$. %
	First, consider $\mathbf{\hat{x}} = (-1,\,0,\,0)$ as a starting point. %
	It can be easily checked that constant control $(u_1,\,u_2) = (1/T,\,0)^T$ drives the system to zero at final time $T$. %
	Now, recall the symmetry group of the robot computed in Proposition~\ref{PropositionSymmetryGroup} %
	and say we choose $g \in \mathcal{G}$ defined by $\Delta \mathbf{x} = (1,\,1,\,0)$. %
	The shifted initial point is then given by $\Psi_g(\mathbf{\hat{x}}) = (0,\,1,\,0)^\top$, see Figure~\ref{fig:ShiftedOCP}.

	Now, the control problem is fundamentally different: $(u_1,\,u_2) = (1/T,\,0)^\top$ does not drive the system to $\mathbf{x}^\star$, %
	in fact we are faced with the famous parallel parking problem, see, e.g.\ \cite{ParoLaug96}. % for holonomic systems.
	A solution could be obtained by a sequence "turn-move-turn", as  presented in~\cite{Wort2015NMPC}, but, at the moment, our focus is on finding invariances. %
	As proposed above, to this aim we modify the terminal condition, i.e.\ the shifted final state becomes $\Psi_g(\mathbf{x}^\star) = (1,\,1,\,0)^\top$. %
	Then, the previously computed solution $(u_1,\,u_2) = (1/T,\,0)^\top$ can drive the system from $\Psi_g(\mathbf{\hat{x}})$ to $\Psi_g(\mathbf{x}^\star)$. %
	Consequently, solutions of the OCP would remain the same, given that the stage costs are either invariant or analogously modified.
\end{example}
\begin{figure}[htb]
	\centering
	\includegraphics[width=.25\textwidth]{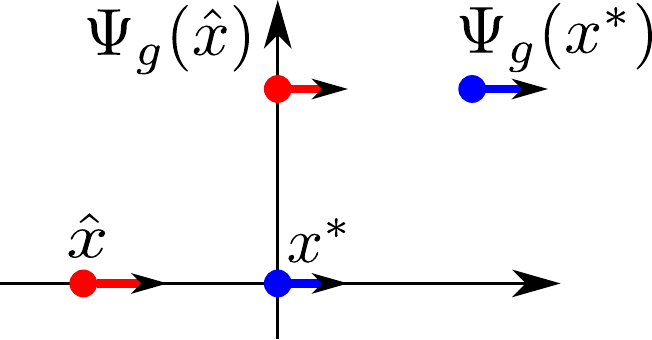}
	\caption{Illustration of the modified terminal constraint for Example~\ref{ExampleMobileRobotModifiedOCP}.}\label{fig:ShiftedOCP}
\end{figure}

\subsection{Optimal Control Problem for the Mobile Robot}\label{SubsectionMobileRobotOCP}

In Proposition~\ref{prob:RobotTrims} we identified characteristic motions of the mobile robot, i.e.~circles and straight lines, which are formally trim primitives. Since any pair $(u_1,u_2)$ of constant control values generates a trim we now switch perspective and define a class of admissible controls which guarantees that the system evolves on a finite sequence of trims. Furthermore, we show that an optimal control function exists for this class of admissible control functions.

Let us suppose that the control has to be piece-wise constant corresponding to the use of at most $S$, $S \in \mathbb{N}$, trims (which do not have to be pairwise different), i.e., for each control function~$u$, there exists a finite partition
\begin{align}\label{NotationPartition}
	0 =t_0 \leq t_1 \leq t_2 \leq \ldots \leq t_S = T
\end{align}	
such that the control function $u|_{[t_{i-1},t_i)}$ is constant for all $i \in \{1,\ldots,S\}$. %
Furthermore, the boundary condition is replaced by an initial and a terminal state constraint, i.e.\ $\mathbf{x}(0) = \hat{\mathbf{x}}$ and $\mathbf{x}(T) = \mathbf{x}^\star$. %
Moreover, instead of the (mixed) control-state constraint $g(\mathbf{x}(t),\mathbf{u}(t))$, %
we consider the state constraint $\mathbf{x}(t) \in \mathbb{X} \subset \mathbb{R}^n$ %
and the control constraint $\mathbf{u}(t) \in \mathbb{U} \subset \mathbb{R}^m$ for a compact and convex set~$\mathbb{X}$ and a compact set~$\mathbb{U}$ with $\mathbf{x}^\star \in \mathbb{X}$. Then, for a given optimization horizon~$T$, $T > 0$, and initial value~$\hat{\mathbf{x}}$, $\hat{\mathbf{x}} \in \mathbb{R}^n$, we define the set of admissible control functions 
\begin{align}\label{eq:admissibleControlFunctions}
	\mathcal{U}^S_T(\hat{\mathbf{x}}) := \left\{ 
		u \hspace*{-0.075cm}:\hspace*{-0.075cm} [0,T] \rightarrow \mathbb{U} \left| \begin{array}{l}
			\exists\ \text{partition~\eqref{NotationPartition}}: u|_{[t_{j-1},t_j)} \equiv \mathbf{u}_j \quad\forall\, j \in \{1,2,\ldots S\} \\ 
			\mathbf{x}(0;\hat{\mathbf{x}},u) = \hat{\mathbf{x}}\text{, }\mathbf{x}(T;\hat{\mathbf{x}},u) = \mathbf{x}^\star\text{ and, $\forall\, t \in [0,T]$,} \\
		\mathbf{x}(t;\hat{\mathbf{x}},u) \in \mathbb{X}\text{, } \dot{\mathbf{x}}(t;\hat{\mathbf{x}},u) = f(\mathbf{x}(t;\hat{\mathbf{x}},u),\mathbf{u}(t)) 
		\end{array} 	
	\right. \hspace*{-0.175cm} \right\}
\end{align}
where the set $\mathbb{U}$ of feasible control values is given by $[-\bar{\mathbf{u}},\bar{\mathbf{u}}] \subset \mathbb{R}^{2}$ with $\bar{u}_i > 0$ for all $i \in \{1,2\}$. Moreover, if the time horizon~$T$ is considered as an optimization variable, the set of admissible control functions is given as $\bigcup_{T > 0} \mathcal{U}_T^{S}(\hat{\mathbf{x}})$. %
Furthermore, for weighting coefficients $c_1 \geq 0$, $c_2 \geq 0$, and $c_3 \geq 0$, %with $\sum_{i=1}^3 c_i = 1$,
let the stage cost $\ell: \mathbb{R}^n \times \mathbb{R}^m \rightarrow \mathbb{R}_{\geq 0}$ by defined as
\begin{align*}
	\ell(\mathbf{x},\mathbf{u}) := c_1 \| \mathbf{u} \|_R^2 + c_2 \vertiii{\mathbf{u}} + c_3
\end{align*}
with $\| u \|_R^2 := u^\top R u$ for a positive definite, symmetric matrix $R$
and $\vertiii{\mathbf{u}}$ an arbitrary norm.

We consider the Optimal Control Problem
\begin{equation}
	\min_{u \in \mathcal{U}^S_T(\hat{\mathbf{x}})} J_T(u,\hat{\mathbf{x}}) \qquad\text{ or }\qquad \min_{T,u \in \bigcup_{T > 0} \mathcal{U}^S(\hat{\mathbf{x}})} J_T(u,\hat{\mathbf{x}}) \label{NotationOcpMinimalTime}
\end{equation}
for $c_3 = 0$ and $c_3 > 0$ respectively. %
Since the stage cost is positive definite w.r.t.\ $\mathbf{u}$, the infimum is bounded from below by zero and labeled~$V_T(\hat{\mathbf{x}})$.

The following lemmata show basic properties of the OCP~\eqref{NotationOcpMinimalTime} for Example~\ref{ExampleMobileRobot}; namely, existence of a feasible solution and finite cost.
\begin{lemma}\label{LemmaAdmissibility}
	Consider Example~\ref{ExampleMobileRobot} and the OCP~\eqref{NotationOcpMinimalTime}. %
	Then, if $S \geq 3$ and the (finite) set of control values contains the values $(0,\ 0)^\top$, $(\bar{u}_1,\ 0)^\top$, and $(0,\ \bar{u}_2)^\top$, %
	there exists, for each $\hat{\mathbf{x}} \in \mathbb{X} \setminus \{\mathbf{x}^\star\}$, a time horizon~$T$, $T > 0$, %
	and a control function $u \in \mathcal{U}^S_T(\hat{\mathbf{x}})$ such that $J_T(\hat{\mathbf{x}},u) < \infty$ holds.
\end{lemma}
\begin{proof}
	Essentially, we can proceed analogously to~\cite{Wort2015NMPC,SchuWort18}. %
	Due to convexity of $\mathbb{X}$ and the fact that the length of the optimization horizon is not fixed, %
	we can simply prolong~$T$ such that the sequence
	\begin{enumerate}
		\item \textit{turn}, i.e.\ use $(0,\bar{u}_2)^\top$ until the robot points towards the origin, 
		\item \textit{move}, i.e.\ use $(\bar{u}_1,0)^\top$ until the robot reaches the origin, and
		\item \textit{turn}, i.e.\ use $(0,\bar{u}_2)^\top$ until the desired orientation is attained,
	\end{enumerate}%
	becomes feasible and ensures that $\mathcal{U}^S_T(\hat{\mathbf{x}})$ is non-empty. %
	Due to compactness of $\mathbb{X}$ and $\mathbb{U}$ and continuity of~$\ell$, this ensures finite costs. If $c_3 = 0$, i.e.\ minimal/short time is not weigthed in the stage cost, the third control value~$\mathbf{0}$ is important because it allows to stay at~$\mathbf{x}^\star$ until the final time~$T$ is reached (which is typically not an optimization variable for $c_3 = 0$). 
\end{proof}

In particular, Lemma~\ref{LemmaAdmissibility} ensures existence and finiteness of the infimum of OCP~\eqref{NotationOcpMinimalTime}. %
Hence, the value function $V: \mathbb{X} \rightarrow \mathbb{R}_{\geq 0}$ is well-defined. %
The following lemma shows that the optimum is attained, see, e.g.\ \cite{lee1967foundations} for a proof.
\begin{lemma}\label{lemma:optimum}
	Consider Example~\ref{ExampleMobileRobot} and the OCP~\eqref{NotationOcpMinimalTime} and let $\hat{\mathbf{x}} \in \mathbb{X} \setminus \{\mathbf{x}^\star\}$ be given. %
	Moreover, let $S \geq 3$ and the (finite) set of control values contain the values $(0,\ 0)^\top$, $(\bar{u}_1,\ 0)^\top$, and $(0,\ \bar{u}_2)^\top$.
	
	For $c_3>0$, there exists a time horizon~$T^\star$, $T^\star > 0$, and a control function $u^\star \in \mathcal{U}^S_{T^\star}(\hat{\mathbf{x}})$ %
	satisfying $J_{T^\star}(\hat{\mathbf{x}},u^\star) = V(\hat{\mathbf{x}})$. 
	
	For $c_3 = 0$ and optimization horizon~$T$ such that $\mathcal{U}^S_T(\hat{\mathbf{x}}) \neq \emptyset$, %
	there exists an admissible control function~$u^\star \in \mathcal{U}^S_{T^\star}(\hat{\mathbf{x}})$ with $J_T(\hat{\mathbf{x}},u^\star) = V(\hat{\mathbf{x}})$.
\end{lemma}

The preceding lemmata justify a restriction to control functions with piecewise constant control values and finitely many switches, %
since existence of admissible control functions can still be guaranteed. %
Restricting to $\mathcal{U}^S_T(\hat{\mathbf{x}})$ is a first step in the direction of a finite maneuver automaton, as proposed by Frazzoli.
In~\cite{FrDaFe05}, two types of motion primitives are distinguished: %
trim primitives, as we defined them in Definition~\ref{def:Trims}, and maneuvers. %
Maneuvers are arbitrarily controlled trajectories that start and end on trims. %
Thus, they allow to concatenate trims and maneuvers alternatingly. The resulting trajectories are admissible to the system dynamics, in particular, switching from trim to maneuver or from a preceding maneuver to a trim is continuous. %
The finite set of motion primitives is called a \emph{maneuver automaton}. %
It can be represented as a graph: Trim primitives are the nodes and edges correspond to maneuvers which have been designed to connect the preceding and succeeding trim.

The robot example is special, since any fixed pair of controls $(u_1,u_2)^{\top}$ leads to a trim %
and no maneuvers are necessary to switch from one trim to another. %
Therefore, we focus on constructing sequences which solely consist of trim primitives. 
\begin{remark}[Sequence of Trim Primitives]
	Any partition of type~\eqref{NotationPartition} %
	together with an arbitrary set of $S$ control values $\mathbf{u}_i = (u^i_1,u^i_2)^\top$ generates a sequence of trim primitives for the mobile robot, which can be transcribed as a trajectory ~$x \in \mathcal{AC}([0,T],M)$ via
\[ 
\mathbf{x}(t;\hat{\mathbf{x}},u) := 
\begin{cases}
\Psi(\exp(\xi_1 t),\hat{\mathbf{x}}) & \text{ for } t_0 \leq t \leq t_1,\\
\Psi(\exp(\xi_2 (t-t_1)),\mathbf{x}(t_1)) & \text{ for } t_1 < t \leq t_2,\\
\vdots\\
\Psi(\exp(\xi_S (t-t_{S-1})),\mathbf{x}(t_{S-1})) & \text{ for } t_{S-1} < t \leq t_S,\\
\end{cases}\]
with $u_i |_{[t_{i-1},t_i)} \equiv (u^i_1,u^i_2)^\top$ for $i=1,\dots,S$, $\xi_i$ being the Lie algebra element defined by $(u^i_1,u^i_2)^\top$ and $\mathbf{x}(t_{i-1};\hat{\mathbf{x}};u)$ according to Proposition~\ref{prob:RobotTrims}, and $\mathbf{x}(t_i)$ being a short notation for  $\mathbf{x}(t_i) := \mathbf{x}(t_i;\hat{\mathbf{x}},u)$.
\end{remark}

\section{Model Predictive Control for the Mobile Robot}\label{sec:sym_MPC}

This section is essentially split into two parts. Firstly, we demonstrate the effectiveness of the techniques proposed in the preceeding two sections by considering Model Predictive Control (MPC) for the mobile robot. To this end, we consider the following MPC scheme where the terminal state~$\mathbf{x}^\star$ is, w.l.o.g., set to zero. 
\begin{algorithm}
	Let $\mathbf{x}^0 \in \mathbb{X} \setminus \{\mathbf{0}\}$ and $\delta>0$ be given.\\
	Set $i = 0$, $t_0 = 0$, and $\hat{\mathbf{x}} = \mathbf{x}^0$.
	\begin{enumerate}
		\item Solve OCP~\eqref{NotationOcpMinimalTime} to compute a minimizer $T^\star$, $u^\star$ and implement $u^\star|_{[0,\min\{\delta,T^\star\})}$.
		\item Set $t_{i+1} := t_i + \min\{\delta,T^\star\}$ and $\hat{\mathbf{x}} := \mathbf{x}(t_{i+1}-t_i;\hat{\mathbf{x}},u^\star)$
		\item If $\hat{\mathbf{x}} = \mathbf{0}$ stop. Otherwise increment $i$ and go to Step~(1)
	\end{enumerate}\label{AlgorithmMpcMinimalTime}
\end{algorithm}

\noindent Algorithm~\ref{AlgorithmMpcMinimalTime} yields the MPC closed-loop trajectory
\begin{align}\label{NotationMpcClosedLoopTrajectory}
	\mathbf{x}^{\operatorname{MPC}}(t_{i+1};\mathbf{x}^0) := 
	\mathbf{x}(\min\{\delta,T^\star_{\mathbf{x}^{\operatorname{MPC}}(t_{i};\mathbf{x}^0)}\};\mathbf{x}^{\operatorname{MPC}}(t_{i};\mathbf{x}^0),u^\star_{\mathbf{x}^{\operatorname{MPC}}(t_{i};\mathbf{x}^0)})
\end{align}
where we emphasized the dependence of the optimal $T^\star$ and $u^\star$ on the initial value $\hat{\mathbf{x}} = \mathbf{x}^{\operatorname{MPC}}(t_i;\mathbf{x}^0)$, which we omitted in the exposition of Algorithm~\ref{AlgorithmMpcMinimalTime}. The concatenated control function is denoted by $u^{\operatorname{MPC}}$. Key properties in MPC are \textit{recursive feasibility} and convergence of the closed-loop trajectory to the desired set point $\mathbf{x}^\star = \mathbf{0}$. The former is important to guarantee that the feasible set of the OCP to be solved in Step~(1) of Algorithm~\ref{AlgorithmMpcMinimalTime} is non-empty provided it was non-empty at $t = 0$ (initial feasibility), see, e.g.\ \cite{RawlingsMayneDiehl2017}. For Algorithm~\ref{AlgorithmMpcMinimalTime}, recursive feasibility is ensured by the terminal constraint, see, e.g.\ \cite{Keerthi1988}.

For the convergence of the closed-loop trajectory, we distinguish whether the coefficient~$c_3$ weighting the \textit{process time} is present in the stage cost or not. If it is, the proof is very simple and presented in Subsection~\ref{SubsectionMinimalTime}. Here, we even prove finite time convergence of the MPC closed-loop trajectory given by~\eqref{NotationMpcClosedLoopTrajectory} to the origin.

Secondly, in Subsection~\ref{sec:qucosts} we give further comments on the usefullness of the proposed techniques if the stage cost is inconsistent with the invariance induced by the symmetry action. Here, we also eschew the terminal equality constraint, which is --~in general~-- not desirable from a numerical point of view.

\subsection{Minimizing Energy and Fuel Consumption} 

We consider the Optimal Control Problem
\begin{equation}
	\boxed{\begin{aligned}
		\text{Minimize}\quad & \int_0^T c_1 \| u(t) \|^2_R + c_2 \vertiii{u(t)}\,\mathrm{d}t 
		\quad\text{w.r.t.\ $u \in \mathcal{U}^S_T(\hat{\mathbf{x}})$}\\
		%& \text{subject to $\mathbf{x}(T) = 0$ and $\mathbf{x}(0) = \hat{\mathbf{x}}$ and~\eqref{NotationMobileRobot}}.
		& \text{with $\mathbf{x}^\star = 0$ and vector field~$f$ given by}~\eqref{NotationMobileRobot}.
	\end{aligned}} \label{OcpMobileRobot}%\tag{OCP}
\end{equation}
Before we rigorously show that the origin is asymptotically stable w.r.t.\ the MPC closed loop (in Theorem \ref{TheoremAsymptoticStability}), %
we establish that, for each optimal solution of the OCP~\eqref{OcpMobileRobot}, the control effort is uniformly distributed on the whole time interval~$[0,T]$ in the following proposition.
\begin{proposition}[Necessary Optimality Condition]\label{PropositionNecessaryOptimalityCondition}
	Let $u^\sharp \in \mathcal{U}^S_T(\hat{\mathbf{x}})$ be an admissible control function for the OCP~\eqref{OcpMobileRobot}. %
	Then, for given weighting coefficients $c_1 > 0$ and $c_2 \geq 0$, $u^\sharp$ either exhibits uniform control effort, i.e.
	\begin{equation}\label{PropertyUniformControlEffort}
		\exists\, c \in \mathbb{R}_{\geq 0}: \|u^\sharp(t)\|_R = c \quad\text{for almost all $t \in [0,T]$}
	\end{equation}
	or we can construct an admissible control function $\bar{u} \in \mathcal{U}^S_T(\hat{\mathbf{x}})$ with %
	\[ \int_0^T c_1 ( \| u^\sharp(t) \|^2_R - \| \bar{u}(t) \|_R^2 ) + c_2 ( \vertiii{ u^\sharp(t) } - \vertiii{ \bar{u}(t) } ) \,\mathrm{d}t > 0,\] i.e.\ %
	strictly smaller objective value.
\end{proposition}
\begin{proof}
	Since $u^\sharp$ is piecewise constant there exists a finite partition
	\[ 
		t_0 := 0 < t_1 < \ldots < t_{S-1} < t_S := T \qquad\text{with}\qquad S \in \mathbb{N}
	\]
	such that $u^\sharp$ is constant on each interval~$(t_{i-1},t_{i})$, $i \in \{1,2,\ldots,S\}$. %
	Assume, w.l.o.g., that the control function $u^\sharp$ exhibits values $u^{(1)}$ on $(t_0,t_1)$ and $u^{(2)}$ on $(t_1,t_2)$ %
	such that $\| u^{(1)} \|^2_R \neq \| u^{(2)} \|_R^2$ holds. Then, we have the costs
	\begin{align}\label{NotationOldCosts}
		t_1 \left( c_1 \| u^{(1)} \|^2_R + c_2 \vertiii{ u^{(1)} } \right) + (t_2-t_1) \left(c_1 \| u^{(2)} \|_R^2 + c_2 \vertiii{ u^{(2)} } \right)
	\end{align}
	on the time interval $[0,t_2)$. %
	In the following, we construct a piecewise constant control %
	such that the corresponding trajectory reaches the same point $\mathbf{x}(t_2;\hat{\mathbf{x}},u^\star)$ at time $t_2$ but produces less costs %
	--- a contradiction to optimality of $u^\sharp$, which shows the claimed assertion. To this end, we exploit the property
	\begin{align}\label{PropertyRaute}
		\mathbf{x}(\alpha t;\bar{\mathbf{x}},u) = \mathbf{x}(t;\bar{\mathbf{x}},\alpha u) 
		\qquad\forall\, (t,\bar{\mathbf{x}}) \in [0,\tau] \times \mathbb{R}^2
	\end{align}
		for arbitrary $\tau > 0$ if $u$ is constant on $(0,\tau)$, see~\cite{SchuWort18}.

	W.l.o.g.\ let $\| u^{(2)} \|_R > \| u^{(1)} \|_R$ hold. Then, we replace $u^{(2)}$ by a scaled version; %
	namely $\alpha u^{(2)}$ with $\alpha \in (0,1)$ sufficiently close to one such that all following quantities are well-defined. %
	Moreover, we enlarge the length of the interval $[t_1,t_2)$ by the factor $\alpha^{-1}$. %
	This implies, using the Identity~\eqref{PropertyRaute}, that the same path ---~starting at $\mathbf{x}(t_1;\hat{\mathbf{x}},u^\sharp)$~--- %
	is traversed in the state space but with a slower speed. %
	Simultaneously, we enlarge $u^{(1)}$ and reduce the length of the respective time interval $[0,t_1)$ such that %
	the same path is traversed and the overall length of both intervals remains unchanged. %
	The latter implies $\alpha^{-1} (t_1 - (1-\alpha)t_2) = t_1/\beta$ if the scaling factor used for $u^{(1)}$ is called~$\beta$. In particular, we get the equations
	\begin{align}\label{EquationKathrin} 
		\beta - 1 = \frac{(1-\alpha)(t_2-t_1)}{t_1-(1-\alpha)t_2}.
	\end{align}
Using the representation of $\beta$ displayed in~\eqref{EquationKathrin} to rewrite the factor $t_1/\beta$ in front of the first summand, we get the costs
\begin{align*}
	\frac {t_1}{\beta} %-(1-\alpha)t_2}{\alpha} 
	\left( \beta^2 c_1 \| u^{(1)} \|^2_R + \beta c_2 \vertiii{ u^{(1)} } \right) + \frac {t_2-t_1}{\alpha} \left( \alpha^2 c_1 \| u^{(2)} \|^2_R + \alpha c_2 \vertiii{ u^{(2)} } \right).
\end{align*}
Subtracting the original value~\eqref{NotationOldCosts} and showing that the resulting expression is strictly less than zero completes the proof. Hence, we have to establish the inequality
\begin{align*}
	t_1 (\beta-1) c_1 \| u^{(1)} \|^2_R < (1-\alpha)(t_2-t_1) c_1 \| u^{(2)} \|_R^2.
\end{align*}
Then, replacing $(\beta - 1)$ using \eqref{EquationKathrin} and dividing by $c_1 (t_2-t_1)(1-\alpha)$ yields
%\begin{align*}
	%t_1 \| u^{(1)} \|^2_R < (t_1 - (1-\alpha)t_2) \| u^{(2)} \|_R^2
%\end{align*}
%or, equivalently,
\begin{align*}
	(1-\alpha)t_2 \| u^{(2)} \|_R^2 < t_1 ( \| u^{(2)} \|_R^2 - \| u^{(1)} \|_R^2).
\end{align*}
	The right hand side is independent of the scaling factor~$\alpha$ and strictly larger than zero, %
	while the left hand side converges to zero for $\alpha \rightarrow 1$. %
	In conclusion, the inequality is satisfied for sufficiently small $\alpha > 0$. 
\end{proof}

The following corollary extends the assertion of Proposition~\ref{PropositionNecessaryOptimalityCondition} to a more general set of admissible control functions, which turns out to be helpful for the proof of the following theorem.
\begin{corollary}\label{CorollaryConstantControlEffort}
	Proposition~\ref{PropositionNecessaryOptimalityCondition} also holds for $u \in \mathcal{L}^{\infty}([0,T],\mathbb{R}^2)$ %
	except on a set of measure zero. 
\end{corollary}
\begin{proof}
	Note that we have not used that there were at most $S$ switches. %
	Hence, the line of reasoning works for arbitrary piecewise control functions %
	--- a class, which is dense in $\mathcal{L}^{\infty}([0,T],\mathbb{R}^2)$, which allows us to conclude the assertion. 
\end{proof}

Proposition~\ref{PropositionNecessaryOptimalityCondition} and its extension formulated in Corollary~\ref{CorollaryConstantControlEffort} are key ingredients to prove convergence of the MPC closed-loop trajectory to the origin %
since they allow us to derive a decrease of the value function. 
\begin{theorem}\label{TheoremAsymptoticStability}
	Let $c_1 > 0$, $c_2 \geq 0$, and an initial state~$\mathbf{x}^0$ be given. %
	Then, we have recursive feasibility for the MPC algorithm~\ref{AlgorithmMpcMinimalTime} and, if $\mathbf{x}^0$ is initially feasible, %
	the corresponding MPC closed-loop trajectory converges to the origin, %
	i.e.\ $\mathbf{x}^{\operatorname{MPC}}(t;x^0)$ exists for all $t \geq 0$ and %
	$\lim_{t \rightarrow \infty} \mathbf{x}^{\operatorname{MPC}}(t;x^0) = \mathbf{0}$ holds.
\end{theorem}
\begin{proof}
	Since the optimal control problem contains a terminal equality constraint, recursive feasibility holds provided that initial feasibility is given. %
	To prove the convergence, %
	note that Proposition~\ref{PropositionNecessaryOptimalityCondition} implies that, %
	for an optimal control function~$u^\star$, $\| \mathbf{u}^\star(t)\|_R^2$ is constant for almost all $t \in [0,T]$. %
	Hence, adding this as a constraint to the optimal control problem, does not change the set of all minimizers (the minimizer). %
	
	If we now modify the optimal control problem by setting $c_2 = 0$ (but leaving $c_1$ as it is), %
	each admissible control function remains admissible and is assigned to an objective value, %
	which is upper bounded by its counterpart of the original OCP. In addition, we further relax the constraints by allowing arbitrary $\mathcal{L}^{\infty}$-functions, see Corollary~\ref{CorollaryConstantControlEffort}. Clearly, the set of minimizers may change. %
	So far, we get the relation
	\begin{align*}
		\int_0^\delta \ell(\mathbf{x}(t;\hat{\mathbf{x}},u^\star),\mathbf{u}^\star(t))\,\mathrm{d}t \geq 
		\frac {\delta \widetilde{V}(\hat{\mathbf{x}})} T
	\end{align*}
	for the optimal control of the original OCP and the optimal value~$\widetilde{V}(\hat{\mathbf{x}})$ of the modified optimal control problem. %
	Then, we can estimate that the control effort is only decreasing for each admissible control function %
	if we replace the cost function by $\lambda_R \| u \|^2$ %
	where $\lambda_R$ denotes the smallest eigenvalue of the positive definite matrix~$R$. %
	Then, dropping the artificially introduced constraint on uniform control effort w.r.t.\ $\|\cdot\|_R^2$, we get
	\begin{align*}
		\int_0^\delta \ell(\mathbf{x}(t;\hat{\mathbf{x}},u^\star),\mathbf{u}^\star(t))\,\mathrm{d}t \geq %
		\frac {\delta \widetilde{V}(\hat{\mathbf{x}})}T \geq \frac {\lambda_r c_1 \delta \widetilde{\widetilde{V}}(\hat{\mathbf{x}})} T
	\end{align*}
	where $\widetilde{\widetilde{V}}(\hat{\mathbf{x}})$ denotes the minimal value of the OCP~\eqref{OcpMobileRobot} with stage costs $\|u(t)\|^2$. %
	Then, further reducing this value by using the simplified dynamics, %
	see Proposition~\ref{PropositionSimplifiedDynamics} in Appendix \ref{section:appendix} (and still using the notation~$\widetilde{\widetilde{V}}$ for the optimal value), %
	we have derived the Lyapunov inequality
	\begin{align}
		V(\mathbf{x}(\delta;\hat{\mathbf{x}},u^\star)) %
		& = V(\hat{x}) - \int_0^\delta \ell(\mathbf{x}(t;\hat{\mathbf{x}},u^\star),\mathbf{u}^\star(t))\,\mathrm{d}t \nonumber \\
		& \leq V(\hat{x}) - \left( \frac {\lambda_R c_1 \delta}{T^2} \right) \cdot \| \hat{\mathbf{x}} \|^2. \nonumber
	\end{align}
	Then, standard arguments, see, e.g.\ \cite{Sontag1998,GrunPann10} can be used to conclude the assertion.
\end{proof}

Next, we extend Proposition~\ref{PropositionNecessaryOptimalityCondition} %
to the case with additional control and state constraints as a preliminary step to show that also the assertions of Theorem~\ref{TheoremAsymptoticStability} remain valid. 
\begin{corollary}\label{CorollaryNecessaryOptimalityCondition}
	Let control constraints $g(\mathbf{u}) \leq \mathbf{0}$ %
	with $g: \mathbb{R}^m \rightarrow \mathbb{R}^p$ be given such that the set %
	$\{ \mathbf{u} \in \mathbb{R}^m: g(\mathbf{u}) \leq \mathbf{0} \}$ is %
	closed, convex, and contains the origin in its interior. %
	Then, if the control function~$u^\sharp$ exhibits neither uniform control effort, i.e.\ \eqref{PropertyUniformControlEffort}, %
	nor satisfies~$\| \mathbf{u}(t) \|_R^2 \geq r^\star$ for almost all $t \in [0,T]$ with the threshold value
	\begin{equation}
		r^\star := \inf \{ r \in \mathbb{R}_{>0} : 
		g(\mathbf{u}) \leq \mathbf{0} \text{ for all $\mathbf{u} \in \mathbb{R}^m$ with $\| \mathbf{u} \|_R^2 \leq r$} \}, 
		\label{DefinitionRstar}
	\end{equation}
	then the alternative proposed in Proposition~\ref{PropositionNecessaryOptimalityCondition} holds, i.e.~$u^\sharp$ is not optimal.
\end{corollary}
\begin{proof}
	The proof is a direct adaptation of the arguments used in the proof of Proposition~\ref{PropositionNecessaryOptimalityCondition} %
	since the proposed construction is still doable as long as there exists an interval, %
	on which the boundary of the control constraints is not yet active %
	(for which reaching/exceeding the threshold value~$r^\star$ is a necessary condition).
\end{proof}

\begin{remark}[State Constraints]
	Note that adding state constraints~$h(\mathbf{x}) \leq \mathbf{0}$ %
	with $h: \mathbb{R}^n \rightarrow \mathbb{R}^q$, $q \in \mathbb{N}$, %
	does not affect the assertions of Proposition~\ref{PropositionNecessaryOptimalityCondition} %
	and Corollary~\ref{CorollaryNecessaryOptimalityCondition} since we have shown the following in the respective proofs: %
	Each path in the $x_1$-$x_2$-plane remains feasible %
	but the optimal time parametrization w.r.t.\ the cost functional is attained only if the proposed necessary optimality condition holds. %
	Therefore, feasibility w.r.t.\ the state constraint set is maintained (the angle is also invariant on a given path).
\end{remark}

Also Theorem~\ref{TheoremAsymptoticStability} remains valid. %
The only changes needed in the proof are the following: %
Firstly, one has to argue that a certain minimal decrease is automatically achieved %
if the condition $\| \mathbf{u}^\star(t) \|_R^2 \geq r^\star$ %
with $r^\star$ defined by \eqref{DefinitionRstar} holds for almost all $t \in [0,T]$. %
Furthermore, dropping the control and state constraints before Proposition~\ref{PropositionSimplifiedDynamics} (see Appendix) is applied, %
leads to the same lower bound and is, thus, doable. Furthermore, note that all results presented in this section also hold if  control functions of class~$\mathcal{L}^{\infty}$ are used instead of $\mathcal{U}^S_T(\cdot)$.
\begin{example} \label{exp:uniformEffort}
We consider again the mobile robot example with states $x_1,x_2,x_3$ and define a control problem from initial state $\hat{\mathbf{x}} =( 0.1,\,1.0,\,0.8)$ to final state $\mathbf{x}^\star = (0,\,0,\, 0)$. Stage costs are defined as 
\[ \ell(\mathbf{x},\mathbf{u}) = 4 u_1^2+u_2^2-3u_1\cdot u_2 + 0.1 \sqrt{u_1^2+u_2^2}. \]
 We transform from Lagrange to Bolza form by introducing a new state $x_4$ and a second auxiliary state $x_5$ by
 \begin{align}
 	\dot{x}_4 & = 4 u_1^2+u_2^2-3u_1\cdot u_2 + 0.1 \sqrt{u_1^2+u_2^2}, \quad x_4(0) = 0 \\
 	 \dot{x}_5 & = 4 u_1^2+u_2^2-3u_1\cdot u_2 , \quad x_5(0) = 0.
 \end{align}
 Then, the cost function is $J = x_4(T)$ and we fix $T=50$.
 A solution is computed numerically by the graphical interface \emph{WORHP Lab} of the optimization software \emph{WORHP} \cite{buskens2012esa,knauer}.
 The robot dynamics are transcribed by the trapezoidal rule on an equidistant time grid with $50$ time points.
  An optimal solution is found after 56 outer-loop iterations of an SQP method and resulting costs are $J= 0.5141$.
  In Figure~\ref{fig:ExampleControlEffort}, the optimal trajectories and controls are shown.
  Our focus is on auxiliary state $x_5$: The optimal control satisfies
  $\Vert \mathbf{u}(t) \Vert_R^2 = \text{const.}$ for all $t \in [0,T]$, which illustrates the result of Proposition~\ref{PropositionNecessaryOptimalityCondition}.
  Thus, the quadratic part of the control effort increases linearly with time.
  Note, however, that the integrated stage cost, i.e.\  $x_4(t)$, does not increase linearly, nor do the individual controls.
\end{example}
\begin{figure}[htb]
	\centering
	\includegraphics[width=.9\textwidth]{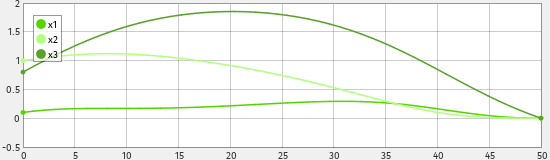}\\[.2cm]
	\includegraphics[width=.9\textwidth]{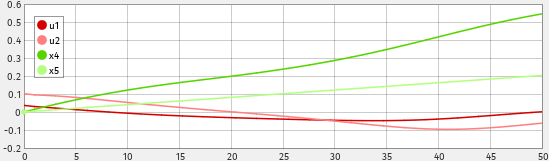}
	 \caption{Optimal solution for the mobile robot as defined in Example~\ref{exp:uniformEffort}: As it has been shown in Proposition~\ref{PropositionNecessaryOptimalityCondition}, the quadratic part of the optimal control effort is uniformly distributed, i.e. $x_5$ increases linearly, although $x_4$ and $u_{1}$, $u_{2}$ do not. 
	 }
	 \label{fig:ExampleControlEffort}
\end{figure}

\subsection{Energy and Fuel Consumption for Finite Sets of Motion Primitives} \label{seq:MPSetting}

Let us now focus on the motion primitives setting. %
That is, we further restrict $\mathbb{U}$ in the definition of the set of admissible control functions~\eqref{eq:admissibleControlFunctions}.
Feasible control values have to belong to a finite set $\{\mathbf{u}^{(1)},\ldots,\mathbf{u}^{(M)}\} \subset [-\bar{\mathbf{u}},\bar{\mathbf{u}}]$, $M \in \mathbb{N}$. Moreover, in accordance with our existence results, see Lemmata~\ref{LemmaAdmissibility} and \ref{lemma:optimum}, we assume that $\mathbf{u}^{(i)} = \mathbf{0}$ holds if and only if $i = 1$.

To this end, let us observe that, for $\hat{x} \in \mathbb{X} \setminus \{\mathbf{0}\}$, %
there exists at least one interval with non-zero control in view of the terminal equality condition. %
Then, we reorder the optimal control such that all indices 
\begin{align*}
	\{ i \in \{1,2,\ldots,S\} : u|_{[t_{i-1},t_i)} \equiv 0\}
\end{align*}
are shifted to the end of the sequence, which can be done without loss of optimality. %
Using this additional condition, we know that we either reach the origin within the sampling interval $[0,\delta)$ or use a control function, %
which is non-zero for each $t \in [0,\delta)$.  %
The latter, however, implies
\begin{align*}
	\int_0^\delta \ell(\mathbf{x}(t;\hat{\mathbf{x}},u^\star_{\hat{\mathbf{x}}}),\mathbf{u}^\star_{\hat{\mathbf{x}}}(t))\, \mathrm{d}t \geq \delta \min_{j \in \{2,\ldots,M\}} c_1 \| \mathbf{u}^{(j)} \|_R^2 + c_2 \vertiii{ \mathbf{u}^{(j)} } =: \tilde{c} \,  > 0,
\end{align*}
which ensures a decrease of at least $\tilde{c}$ in each MPC step. Since $V(\mathbf{x}^0)$ is finite, we get finite time convergence.

\subsection{Penalization of the Process Time}\label{SubsectionMinimalTime}

In this subsection, we present a convergence proof for stage costs, in which the process time is penalized. %
While convergence is clear, %
MPC may contribute to further reduce the costs while still ensuring finite time convergence as shown in the following proposition.
\begin{proposition}
	Consider the OCP~\eqref{NotationOcpMinimalTime} with $c_3 > 0$ and let $\mathbf{x}^0 \in \mathbb{X} \setminus \{\mathbf{0}\}$ be given. If the OCP~\eqref{NotationOcpMinimalTime} is initially feasible, i.e.\ if there exists a time~$T$ and a control function~$u$, $u \in \mathcal{U}_{T}^{S}$, the MPC closed-loop trajectory is well-defined. Moreover, there exists a time~$T^\sharp$, $T^\sharp \in (0,\infty)$, such that $\mathbf{x}^{\operatorname{MPC}}(T^\sharp;\mathbf{x}^0) = \mathbf{0}$ holds.
\end{proposition}
\begin{proof}
	We use the abbreviations $\hat{\mathbf{x}} := \mathbf{x}^{\operatorname{MPC}}(t_{i};\mathbf{x}^0)$ and $\Delta t := t_{i+1}-t_i$. Recursive feasibility can be directly concluded from the admissibility of the shifted control sequence $u^\star_{\hat{\mathbf{x}}}(\cdot + \Delta t)$ at the successor time instant. Moreover, we have
	\begin{align*}
		V(\hat{\mathbf{x}}) & = \int_0^{\Delta t} \ell(\mathbf{x}(t;\hat{\mathbf{x}},u^\star_{\hat{\mathbf{x}}}),\mathbf{u}^\star_{\hat{\mathbf{x}}}(t))\, \mathrm{d}t + J_{T^\star_{\hat{\mathbf{x}}}-\Delta t}(\mathbf{x}(\Delta t;\hat{\mathbf{x}},u^\star_{\hat{\mathbf{x}}}),u^\star_{\hat{\mathbf{x}}}(\cdot+\Delta t)) \\
		& \geq c_3 \Delta t + V(\mathbf{x}(\Delta t;\hat{\mathbf{x}},u^\star_{\hat{\mathbf{x}}}))
	\end{align*}
	If there exists an index~$i$ such that $V(\mathbf{x}^{\operatorname{MPC}}(t_i;\mathbf{x}^0))=0$ holds, we are done. Otherwise, taking into account that $V$ is positive definite, we get
	\begin{align*}
		V(\mathbf{x}^{\operatorname{MPC}}(t_i;\mathbf{x}^0)) \leq V(\mathbf{x}^0) - i \delta c_3
	\end{align*}
	using a telescope sum argument. Then, for $i \rightarrow \infty$, the term $i \delta c_3$ grows unboundedly, which implies that the right hand side becomes smaller than zero for sufficiently large $i$ (e.g., $i := \lceil V(\mathbf{x}^0) (\delta c_3)^{-1} \rceil$) --- a contradiction. 
\end{proof}

\subsection{MPC without Terminal Constraint and Outlook}\label{sec:qucosts}

Here, we want to highlight that motion primitives (or trims if the terminology is supposed to be solely adapted to the example of the mobile robot) were already used in~\cite{Wort2016CST} and the follow-up paper~\cite{Wort2015NMPC} to rigorously ensure asymptotic stability for the setting in which neither terminal constraints nor terminal costs were used. The key idea was to derive the controllability assumption initially proposed by Tuna et al.\ in~\cite{TunaMessinaTeel2006} in combination with the suboptimality estimates from~\cite{GrunPann10}, see also~\cite{Reble_Allgower_2012} and~\cite{Wort2014SICON} for the extension to the continuous-time setting. In~\cite{Wort2016CST}, bounds on the value function in dependence of the initial condition were deduced by using the simple sequence \textit{turn-move-turn} as explicated in Subsection~\ref{SubsectionMobileRobotOCP}. A key element was to use a parametric representation of the solution trajectory, which nicely corresponds to the explanation provided in Section~\ref{sec:Preliminaries}. 

In conclusion, combining the blueprint outlined in~\cite{Wort2015NMPC,Wort2016CST} and the wording and deeper insight in the use of motion primitives to quantize the nonlinear system dynamics seems to be a very promising approach to tackle systems, for which the linearization does not contain sufficient information to fulfill the stabilization task. Here, it is worth mentioning that purely quadratic costs do, in general, not work for the example of the mobile robot, see~\cite{MullWort17}. %
Moreover, the proposed symmetry exploiting technique can also be used to verify initial feasibility, to characterize a set of initially feasible states, and to rigorously treat obstacle avoidance problems.

\section{Predictive Control based on Trim Primitives: Numerical Results}\label{sec:numerics}

Numerical results for the mobile robot example are shown to illustrate the effect of quantizing the set of control values to trim primitives.

\subsection{Quantization of the Set of Feasible Control Values}\label{SubsectionSimpleExample}

Let the initial value $\mathbf{x}^0 = (-2,0,0)^\top$, the terminal set $\mathbb{X} = \{ (0,0,0)^\top \}$, and the stage cost $\ell(\mathbf{x},\mathbf{u}) = \| \mathbf{u} \|^2$ be given. Moreover, we use the set
\begin{align}\nonumber
	\mathbb{U} = \left\{ \mathbf{u} = \left(\begin{array}{c} u_1 \\ u_2 \end{array}\right) \in [-2,2]^2\ \left|\ \exists (j,k) \in \mathbb{Z}^2 \text{ : } \mathbf{u} = \Delta u \left( \begin{array}{c} j \\ k \end{array}\right) \right. \right\}\\
\end{align}
with $\Delta u=0.1$. %
Then, the minimal optimization horizon~$T$ such that initial feasibility is ensured is $T = 1$. In the following, we use the time shift $\delta = 0.1$. Moreover, if the optimal control function is not unique, we choose a sequence with maximal costs on the interval~$[0,\delta)$:
\begin{enumerate}
	\item At time $t = 0$, we get $u^\star \equiv (2, 0)^\top$ and $V(\mathbf{x}^0) = 4$. Hence, we have $\mathbf{x}^{\text{MPC}}(\delta;\mathbf{x}^0) = \mathbf{x}(\delta;\mathbf{x}^0,u^\star) = (-1.8,0)^\top$ and the closed-loop costs given by $\int_0^\delta \ell(\mathbf{x}(t;\mathbf{x}^0,u^\star),\mathbf{u}^\star(t))\,\mathrm{d}t = 0.4$.
	\item At time $t = \delta$, we get $u^\star = u^\star(\mathbf{x}^{\text{MPC}}(\delta;\mathbf{x}^0)) \equiv (1.8, 0)^\top$, $x^{\text{MPC}}(2 \delta;\mathbf{x}^0) = (-1.62,0)^\top$, and $V(\mathbf{x}^{\text{MPC}}(\delta;\mathbf{x}^0)) = 3.24$. Hence, we have reduced the overall costs from $4.00$ to $0.4 + 3.24 = 3.64$. The closed-loop costs on $[0,2\delta)$ are $0.4+0.324=0.724$.
	\item At time $t = 2 \delta$, we get $\mathbf{u}^\star(t) = (1.7,0)^\top$ on $[0,0.2)$ and $\mathbf{u}^\star(t) = (1.6,0)^\top$ for $t \in [0.2,1)$ (using our convention since $u^\star$ is not unique). 
\end{enumerate}	
The following values are summarized in Table~\ref{tab:SimpleExample}.
\begin{table}[htb]
	\begin{tabular}{|c||c|c|c|c|c|c|}\hline
		$i$ & $t = i\delta$ & $\mathbf{x}^{\text{MPC}}_1(t;\mathbf{x}^0)$ & $u_1^\star(0)$ & $V(x^{\text{MPC}}(t;\mathbf{x}^0))$ & $\ldots + \int_0^{t} u^{\text{MPC}}(t)^2\,\mathrm{d}t$ \\ \hline\hline
		0 & 0.0 & -2.00 & 2.0 & 4.00 & 4.000 \\ 
		1 & 0.1 & -1.80 & 1.8 & 3.24 & 3.640 \\
		2 & 0.2 & -1.62 & 1.7 & 2.626 & 3.350 \\
		3 & 0.3 & -1.45 & 1.5 & 2.105 & 3.118 \\ 
		4 & 0.4 & -1.30 & 1.3 & 1.690 & 2.928 \\
		5 & 0.5 & -1.17 & 1.2 & 1.371 & 2.778 \\
		6 & 0.6 & -1.05 & 1.1 & 1.105 & 2.656 \\
		7 & 0.7 & -0.94 & 1.0 & 0.886 & 2.558 \\
		8 & 0.8 & -0.84 & 0.9 & 0.708 & 2.480 \\
		9 & 0.9 & -0.75 & 0.8 & 0.565 & 2.418\\
		10 & 1.0 & -0.67 & 0.7 & 0.451 & 2.368 \\	
		\vdots & \vdots & \vdots & \vdots & \vdots & \vdots \\
		20 & 2.0 & -0.21 & 0.3 & 0.045 & 2.192\\
		\vdots & \vdots & \vdots & \vdots & \vdots & \vdots \\
		35 & 3.5 & \phantom{-}0.00 & 0.0 & 0.000 & 2.182\\ \hline
	\end{tabular}
	\caption{Development of the value function and the MPC closed-loop costs for the example presented in Subsection~\ref{SubsectionSimpleExample}.}
	\label{tab:SimpleExample}
\end{table}

The closed-loop cost for reaching the origin are significantly less than the costs associated to the OCP at time $t=0$ ($2.182$ in comparison to $4.000$). The reason is that the control effort is constantly reduced by using MPC since there is some additional freedom to satisfy the terminal equality constraint after each MPC iteration. If a coarser discretization, e.g.\ $\Delta u = 0.5$, and --~as a consequence~-- a smaller trim library is used, the origin is reached after $20$ steps and the closed-loop costs are $3.000$. In conclusion, there is a trade-off between the numerical effort for solving the combinatorial OCP online (which is drastically increasing for a refined quantization) and the closed-loop performance.

\subsection{Optimal Control with Trim Primitives}\label{sec:numericsOCP}

We consider again the dynamics of the mobile robot, cf.\ Example~\ref{ExampleMobileRobot}, to illustrate qualitatively different optimal solution built of trim primitive sequences.
The parallel parking problem from $x^0 = (0,1,0)^{\top}$ to $x^\star = (0,0,0)^{\top}$ shall be solved in $T=8.0$ time units by optimization w.r.t.\ various cost functionals.

We restrict to the library of trim primitives as given in Table~\ref{tab:primitives}, stored as tuples $(u_1,u_2)^{\top}$ and we consider sequences with at most 4 switches.
Since algorithmic performance is not in the focus of this work, we globally search through all possible combinations of trim primitives and compute the optimal switching times in each case for which a solution can be found.
Note that sequences with fewer switches can be found since two succeeding primitives might be identical.
The rest trim plays an important role so that solutions which would not need $T=8$ time units can be prolonged so that they become feasible.

\begin{table}
\centering
\begin{tabular}{|l|l|l|}
\hline
No. & $(u_1,u_2)^{\top}$ & Trim primitive\\
\hline
1 & $(0,0)^{\top}$ & rest \\
2 & $(1.5,0)^{\top}$ & move straight \\
3 & $(-1/4,-1)^{\top}$ & circle clockwise \\
4 & $(-1/4,1)^{\top}$ & circle anti-clockwise \\
5 & $(0,1)^{\top}$ & turn on the spot\\
\hline
\end{tabular}
\caption{Library of trim primitives for numerical tests in Section~\ref{sec:numericsOCP}.}
\label{tab:primitives}
\end{table}

\begin{figure}
\includegraphics[height=5.5cm]{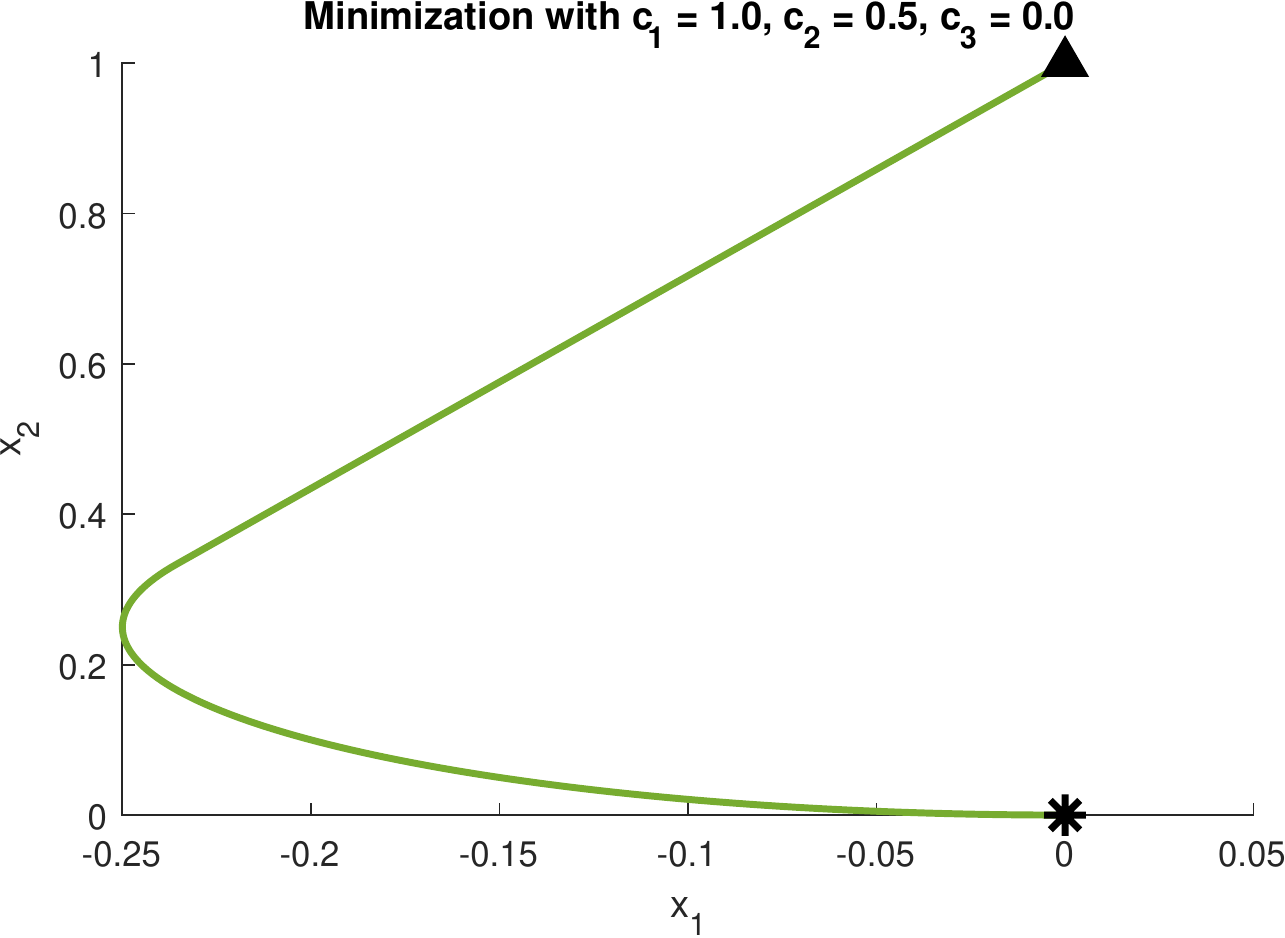}
\includegraphics[height=5.5cm]{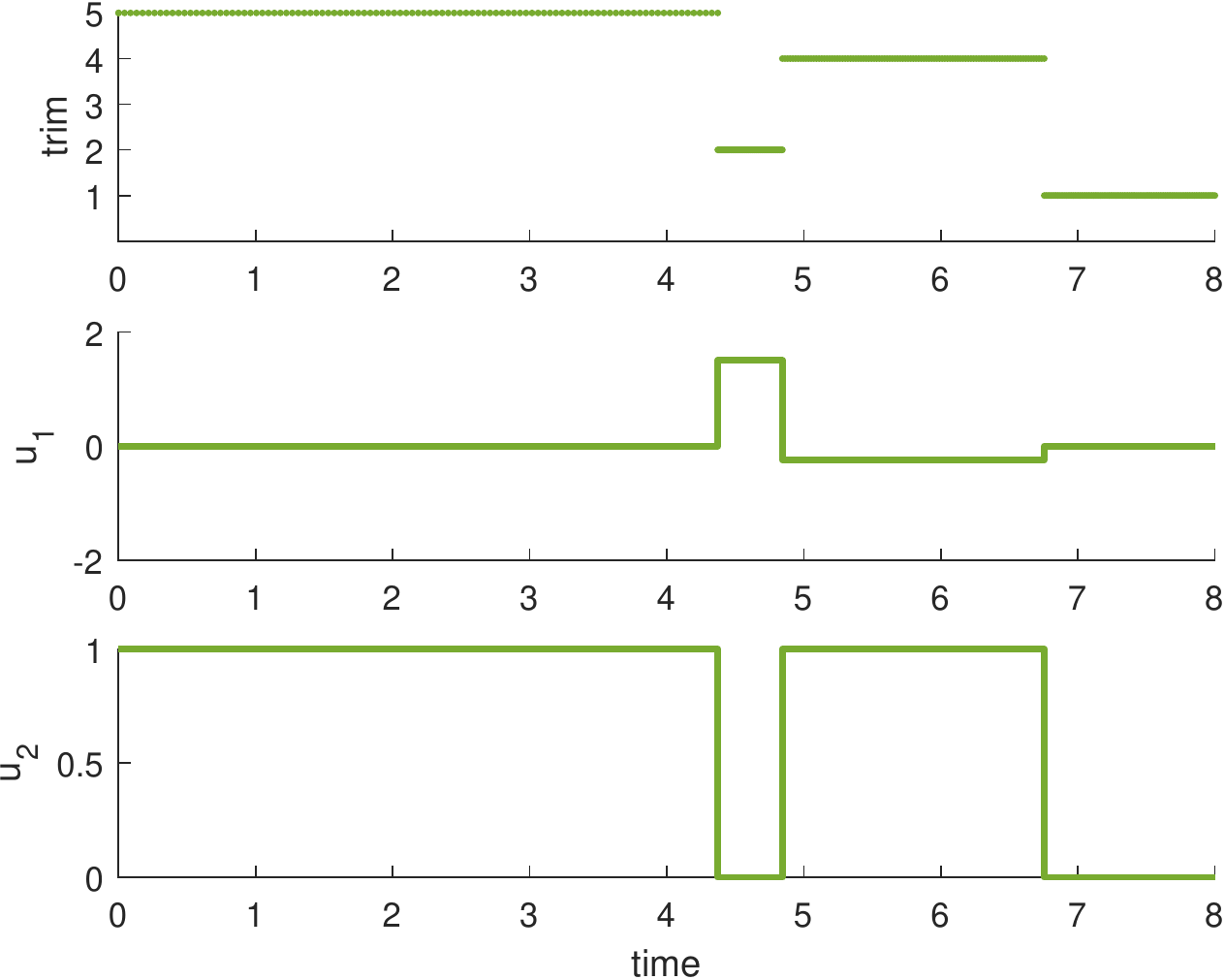}
\caption{OCP solution for the parallel parking problem of the mobile robot in $T=8$ while minimizing $\ell = \vert| u \vert|^2_I + \frac{1}{2} \vert| u \vert|_I$ for the best sequences of trim primitives obtained from a global search on possible sequences.}
\label{fig:numericsBestSolution}
\end{figure}

\begin{figure}
\includegraphics[height=5.5cm]{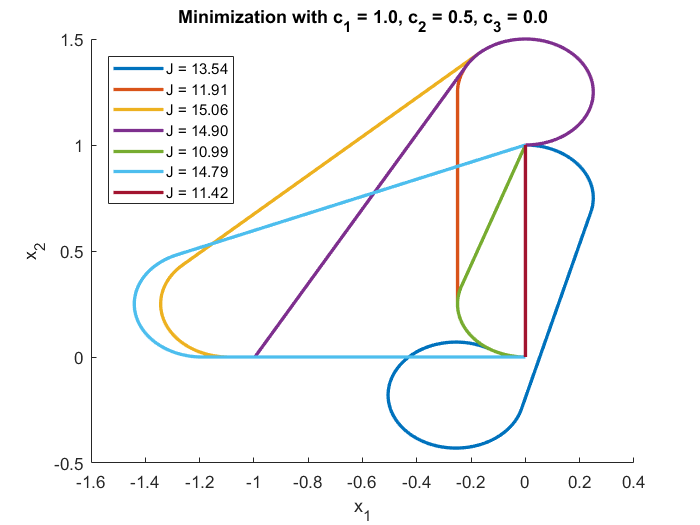}
\includegraphics[height=5.5cm]{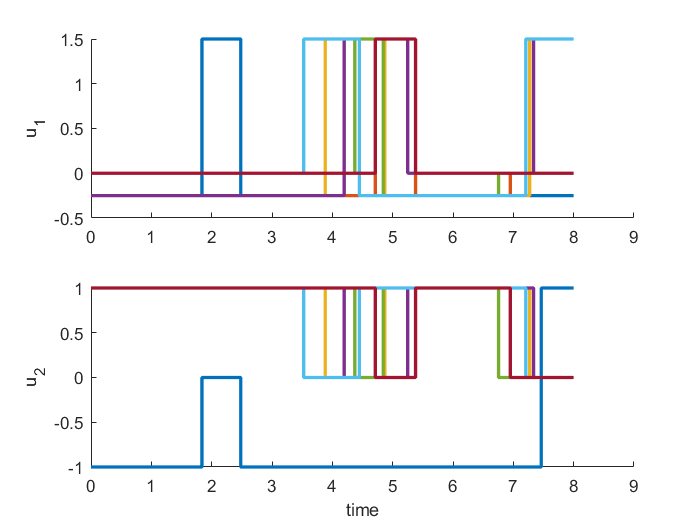}
\caption{OCP solutions for the parallel parking problem of the mobile robot in $T=8$ while minimizing $\ell = \vert| u \vert|^2_I + \frac{1}{2} \vert| u \vert|_I$ for different sequences of trim primitives.}
\label{fig:numericsL1L2}
\end{figure}

We show the solutions for minimizing the cost
$ \ell(\mathbf{x},\mathbf{u}) = \Vert \mathbf{u} \Vert^2_I + 0.5 \cdot \Vert  \mathbf{u} \Vert_2$.
The best solution is given in Figure~\ref{fig:numericsBestSolution}.
It uses the trim sequence $(5,2,4,1)$ (cf.\ Table~\ref{tab:primitives}), as can be seen in the right top picture.
Alternative solutions, also with at most 4 switches, are given in 
Figure~\ref{fig:numericsL1L2}.
When searching for time minimal solutions, the results depicted in Figure~\ref{fig:numericsTime} are obtained.
The control curves code the switching sequence of the solution.
Note that the state plot is a projection to the $(x_1,x_2)$-plane, i.e.\ non-smooth turns (edges) in trajectory have in fact a turning phase, such that the mobile robot does fulfill its nonholonomic constraints.

It can be seen that even a small library of primitives can generate different types of solutions.
In Figure~\ref{fig:numericsL1L2}, the red solution is the \emph{turn-move-turn} sequence.
However, there exists a solution with lower costs (green).
Other solutions have much higher costs and would not be chosen in the unconstrained scenario.
However, they might become of importance as soon as obstacle avoidance is included in the problem.

In Figure~\ref{fig:numericsTime}, one can see that the optimal time is depending on the considered sequence of trim primitives.

\begin{figure}
\includegraphics[height=5.5cm]{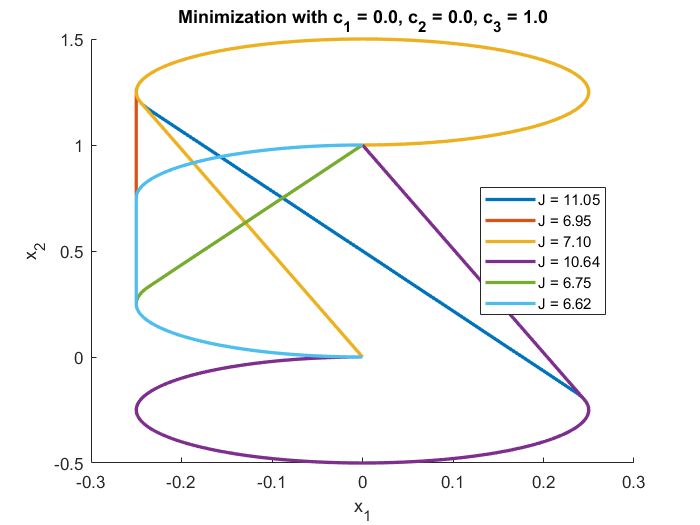}
\includegraphics[height=5.5cm]{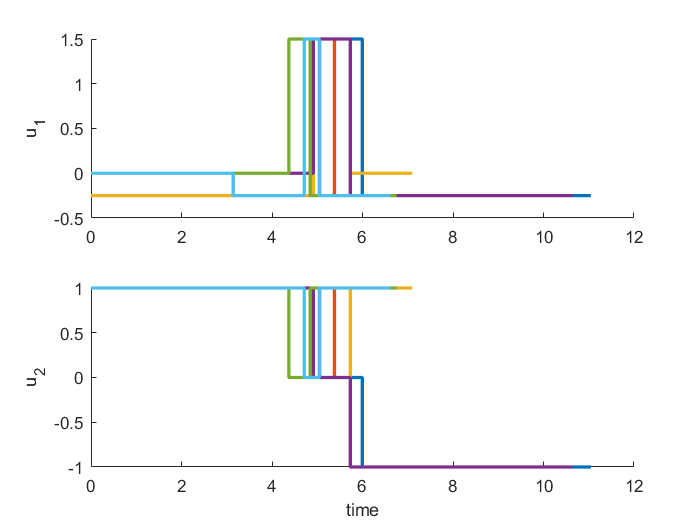}
\caption{Time optimal solution for the parallel parking problem of the mobile robot for different sequences of trim primitives.}
\label{fig:numericsTime}
\end{figure}

\subsection{MPC with Trim Primitives}\label{sec:numericsMPC}

We keep the library of trim primitives which was chosen    in the previous subsection to solve the parallel parking problem by an MPC scheme now.
Here we compute $12$ MPC steps and set $\delta=1$.
Exemplarily, we show the solution for minimizing $\ell = |u|^2$ with fixed horizon $T=8$ in every MPC step in Figure~\ref{fig:numericsMPC}.
The MPC scheme is able to stabilize the system in $8$ time instances.
In every MPC step, the cost decreases.
Additionally, at $t=2$ and $t=3$ a \emph{replanning} occurs, i.e.\ another sequence of motion primitives becomes more efficient than the old solution. (Note that a previous solution can always be prolonged to a valid new solution with the help of the rest trim.)
In the future, we would like to investigate the interplay between quantization and closed-loop performance.
Larger libraries tend to higher computationally costs but potentially to a better closed-loop performance.
Thus, one is interested in the trade-off of these two conflicting optimization goals.

\begin{figure}
\includegraphics[height=5.5cm]{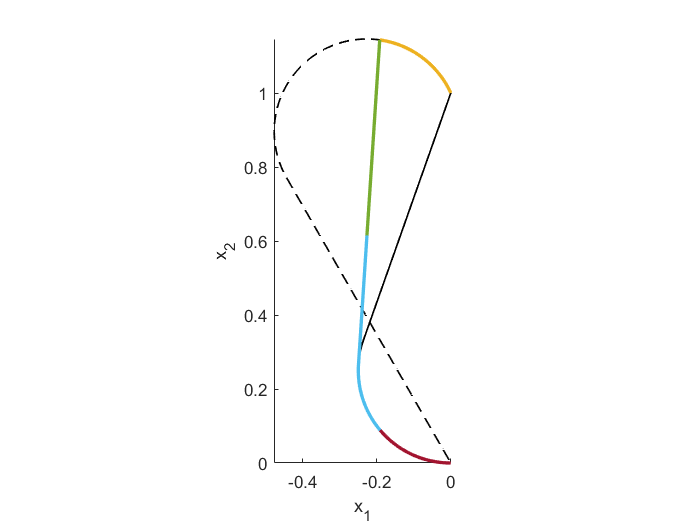}
\includegraphics[height=5.5cm]{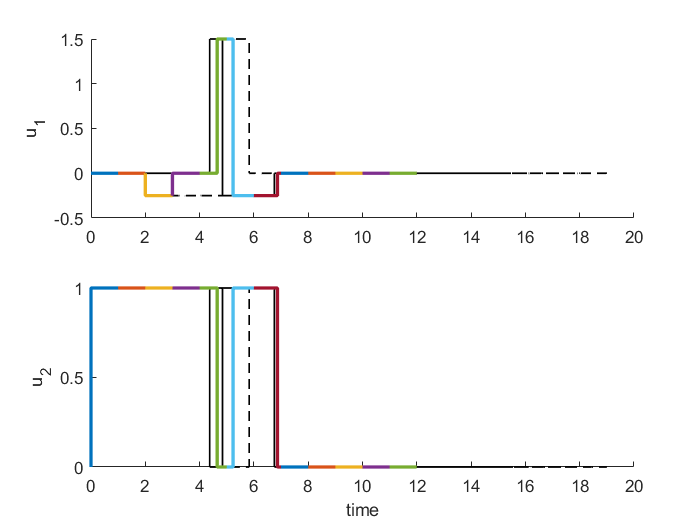}\\
\includegraphics[height=5.5cm]{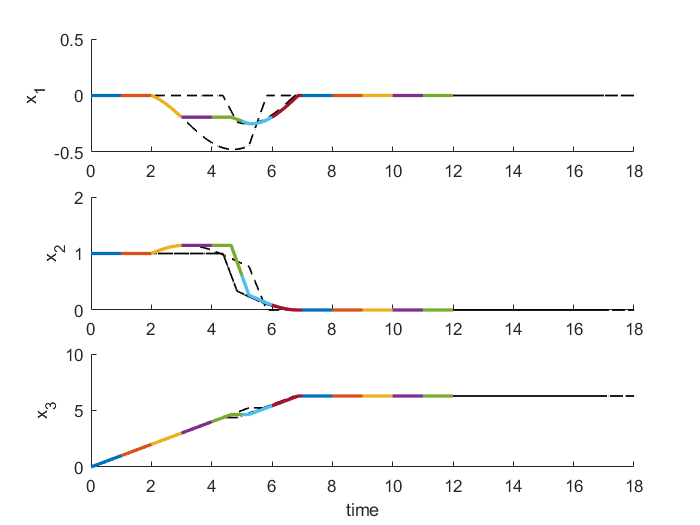}
\includegraphics[height=5.5cm]{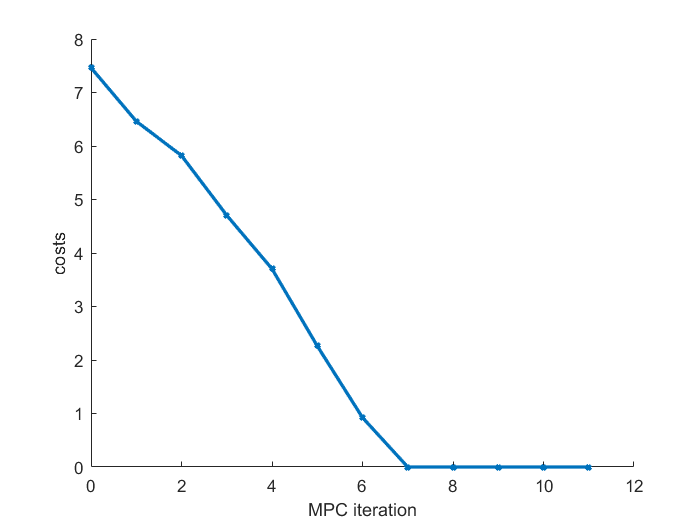}
\caption{MPC solution for the parallel parking problem of the mobile robot in $T=8$ with MPC steps of $\delta =1$. After six iterations, the final point is reached. The optimal sequence of trim primitives is replanned at $t=2$ and $t=3$.}
\label{fig:numericsMPC}
\end{figure}

\section{Conclusions}\label{sec:conclusions}

We propose to exploit inherent symmetries by using motion primitives in the design, analysis, and numerical treatment of MPC schemes. %
%We have first revisited symmetries and motion primitives. Here, t
The key advantage in doing so is that trajectories for genuine nonlinear systems can be easily represented by one-parameter groups. %
W.r.t.\ the design and analysis of MPC schemes, we in particular re-interpreted the results presented in~\cite{Font01,GuHu05} (with terminal constraints and costs) and \cite{Wort2016CST} (without terminal constraints and costs) to lay the foundation for their generalization to a larger class of systems. %
Hereby, it is important to check consistency of stage costs and constraints w.r.t.\ the symmetries of the system dynamics. %
Regarding algorithmic aspects, the quantization of the state space by choosing tailored maneuvers encoded by motion primitives is an essential and helpful concept to encounter, on the one hand, the curse of dimensionality in dynamic programming. %
On the other hand, nonlinear MPC is typically realized via local optimization methods. Thus, getting stuck in local optima is a common problem which can be circumvented by finding (approximations to) alternative solutions via globally searching on a motion primitive graph, cf.\ e.g. \cite{KaraFraz11}.
A potential next step to further enhance the compatability of motion primitives and MPC is to take tailored numerical techniques, see, e.g.\ \cite{KalaGupt17}, into account.

We studied the representative example of the mobile robot in depth in order to illustrate our findings. 
Here, we derived new necessary optimality conditions for the open-loop OCP and provided numerical simulations to shed some light on the trade-off between performance and numerical effort, which corresponds to setting up a suitable library of motion primitives and solving the respective mixed integer OCP.

In conclusion, motion primitives seem to be a (very) promising approach to systematically verify stability conditions like cost controllability as outlined in \cite{Wort2015NMPC} without using the proper wording. Furthermore, trims correspond to inherent optimality/turnpike properties of trims, which are useful for the structural analysis of optimal control problems, see, e.g.\ \cite{FaulFlas19}. Moreover, the proposed combination of motion primitives and MPC is also beneficial if (moving) obstacles or potentially non-convex constraints have to be considered as, e.g., in distributed MPC and for the efficient construction of alternative solutions in order to avoid local minima in the numerical solution of the OCP to be solved in each MPC step.

\appendix
\section*{Appendix}
\section{OCP with Simplified Dynamics}\label{section:appendix}

In this section, we consider an auxiliary OCP, which is needed in order to prove our main result Theorem~\ref{TheoremAsymptoticStability}. The auxiliary OCP is constrained by the system dynamics~\eqref{NotationMobileRobot}, in which the $x_3$-dependence of the first two components is replaced by the additional control~$u_3$. Moreover, $u_3$ is not penalized in the objective function.
\begin{proposition}\label{PropositionSimplifiedDynamics}
	We consider the optimal control problem
	\begin{equation*}%\begin{equation}\boxed{\begin{aligned}
		\text{Minimize}\quad \int_0^T u_1(t)^2 + u_2(t)^2\,\mathrm{d}t
	\end{equation*}
	w.r.t. $u \in \mathcal{L}^\infty([0,T],\mathbb{R}^3)$ subject to the boundary conditions $\mathbf{x}(0) = \hat{\mathbf{x}}$, $\mathbf{x}(T) = 0$ and, for almost all~$t \in [0,T]$, the differential equation
	\begin{equation*}
		\dot{\mathbf{x}}(t) = \begin{pmatrix}
			\cos (u_3(t)) \\ \sin (u_3(t)) \\ 0
		\end{pmatrix} u_1(t) + \begin{pmatrix}
			0 \\ 0 \\ 1
		\end{pmatrix} u_2(t)
	\end{equation*} %\end{aligned}} \nonumber \end{equation}
	Then, the value function $\tilde{\tilde{V}}(\hat{x}): \mathbb{R}^3 \rightarrow \mathbb{R}$ is given by the positive definite function
	\begin{equation}
		\tilde{\tilde{V}}(\hat{x}) = \| \hat{x} \|^2 / T. \nonumber
	\end{equation}
\end{proposition}
\begin{proof}
	Since the OCP in consideration is decoupled, we can split it into the following two optimal control problems:
	\begin{equation}
		\boxed{\begin{aligned}
			\text{Minimize}\quad & \int_0^T u_1(t)^2\,\mathrm{d}t \quad\text{w.r.t.\ $u_i \in \mathcal{L}^\infty([0,T],\mathbb{R})$, $i \in \{1,3\}$,}\\
			& \text{subject to $x_i(T) = 0$ and $x_i(0) = \hat{x}_i$, $i \in \{1,2\}$, and} \\
			& \qquad \begin{pmatrix} \dot{x}_1(t) \\ \dot{x}_2(t) \end{pmatrix} = \begin{pmatrix}
				\cos (u_3(t)) \\ \sin (u_3(t))
			\end{pmatrix} u_1(t).%, \quad \begin{pmatrix} x_1(0) \\ x_2(0) \end{pmatrix} = \begin{pmatrix} \hat{x}_1 \\ \hat{x}_2 \end{pmatrix}.
		\end{aligned}} \label{SimplifiedOCPu1u3}\tag{OCP 1}
	\end{equation}
	\begin{equation}\boxed{\begin{aligned}
		\text{Minimize}\quad & \int_0^T u_2(t)^2\,\mathrm{d}t \quad\text{w.r.t.\ $u_2 \in \mathcal{L}^\infty([0,T],\mathbb{R})$} \\
		& \text{subject to $\dot{x}_3(t) = u_2(t)$ and $x_3(0) = \hat{x}_3$, $x_3(T) = 0$}.
		\end{aligned}} \label{SimplifiedOCPu2}\tag{OCP 2}
	\end{equation}
	Firstly, we solve \eqref{SimplifiedOCPu1u3}: We assume $\min \{|\hat{x}_1|,|\hat{x}_2|\} > 0$, i.e.\ a \textit{non-zero initial condition}. %
	Otherwise the assertion holds trivially since the stage cost is bounded from below by zero. %
	The Hamiltonian is given by
	\begin{align*}
		\mathcal{H}(x_1,x_2,\lambda_1,\lambda_2,u_1,u_3) & = \lambda_0 u_1^2 + (\lambda_1 \cos (u_3) + \lambda_2 \sin (u_3)) u_1.
	\end{align*}
	Differentiation of the Hamiltonian~$\mathcal{H}$ w.r.t.\ the state variables $x_1$, $x_2$ yields %
	the adjoint equation $\dot{\mathbf{\lambda}}(t) = \mathbf{0}$, i.e.\ the adjoints $\lambda_1$ and $\lambda_2$ are constant. %
	Moreover, Pontryagin's maximum principle also yields the necessary optimality conditions
	\begin{enumerate}
		\item [(1)] %
			$\mathcal{H}_{u_1} = 0 \Longleftrightarrow \lambda_0 u_1^\star(t) = - \frac 12 \Big( \lambda_1 \cos(u_3^\star(t)) + \lambda_2 \sin(u^\star_3(t)) \Big)$
		\item [(3)] %
			$\mathcal{H}_{u_3} = 0 \Longleftrightarrow u_1^\star(t) \Big( \lambda_2 \cos(u^\star_3(t)) - \lambda_1 \sin(u^\star_3(t)) \Big) = 0$
	\end{enumerate}
	Firstly, we observe that the right hand side of the differential equation is equal to zero if $u_1^\star(t) = 0$ holds. %
	Combining this observation with the assumed non-zero initial condition and the imposed terminal constraint implies that the set
	\begin{equation}
		S := \{ t \in [0,T] : u^\star_1(t) \neq 0 \} \nonumber
	\end{equation}
	has strictly positive measure~$|S|$. 
	
	Next, we show that $\lambda_0 \neq 0$, which allows us to set $\lambda_0 := 1$ w.l.o.g.\ in the following: %
	Suppose that $\lambda_0 = 0$ holds. Moreover, let us assume that also $\lambda_1$ equals zero. %
	Then, Conditions~(1) and~(3) imply $\sin(u_3^\star(t)) = 0 = \cos(u_3^\star(t))$ in view of $\lambda_2 \neq 0$ for all $t \in S$ %
	--- a contradiction since $S$ has non-zero measure. %
	Analogously, we also get a contradiction for $\lambda_2 = 0$. Hence, we have $\lambda_1 \neq 0 \neq \lambda_2$. %
	Then, Conditions~(1) and~(3) imply $\tan(u_3^\star(t)) = - \lambda_1 / \lambda_2$ and $\tan(u_3^\star(t)) = \lambda_2 / \lambda_1$. %
	Combining these two equations yields $-\lambda_1^2 = \lambda_2^2$ --- again a contradiction.
	Thus, let $\lambda_0 = 1$ in the following.
	
	For each $t \in S \subseteq [0,T]$, Condition~(3) yields %
	%\begin{equation}
		$\lambda_1 \sin(u_3^\star(t)) = \lambda_2 \cos(u_3^\star(t))$, %\nonumber
	%\end{equation}
	which implies
	\begin{equation}
		u^\star_3(t) = \begin{cases} 
			\frac \pi 2 + k \pi \quad\text{ for some $k \in \mathbb{Z}$} & \text{for $\lambda_1 = 0$} \\
			\arctan \left( \frac {\lambda_2}{\lambda_1} \right) & \text{for $\lambda_1 \neq 0.$}
		\end{cases} \label{NotationU3}
	\end{equation}
	Then, the terminal conditions read
	\begin{align*}
		0 = x_1(T) & = \hat{x}_1 + \int_S \cos(u_3^\star(t)) u_1^\star(t)\, \mathrm{d}t, \\
		0 = x_2(T) & = \hat{x}_2 + \int_S \sin(u_3^\star(t)) u_1^\star(t)\, \mathrm{d}t,
	\end{align*}
	which can be rewritten as
	\begin{align*}
		2 \hat{x}_1 & =  \int_S \lambda_1 \cos^2(u^\star_3(t)) + \lambda_2 \cos(u_3^\star(t)) \sin(u_3^\star(t))\, \mathrm{d}t, \\
		2 \hat{x}_2 & =  \int_S \lambda_1 \cos(u_3^\star(t)) \sin(u_3^\star(t)) + \lambda_2 \sin^2(u^\star_3(t))\, \mathrm{d}t
	\end{align*}
	by using Condition~(1). Then, plugging \eqref{NotationU3} into these equations leads to
	\begin{align*}
		& \lambda_i = \frac {2\hat{x}_i}{|S|} \qquad\text{for $i \in \{1,2\}$}
	\end{align*}
	%for $\lambda_1 \neq 0$ 
	using the formulas 
	\[	
		\cos \arctan(x) = \frac 1 {\sqrt{1+x^2}} \qquad\text{and}\qquad \sin \arctan(x) = \frac x {\sqrt{1+x^2}}.
	\]
	%and $\hat{x}_1 = 0$ and $\lambda_2 = 2\hat{x}_2 / |S|$ for $\lambda_1 = 0$.
	Necessarily, this leads either to $\hat{x}_1 = 0$ for $\lambda_1 = 0$ or to a contradiction otherwise. Consequently, we get
	\begin{align*}
		u_1^\star & = - \frac 12 \left( \frac {2\hat{x}_1}{|S|} \cos \arctan \left( \frac {\hat{x}_2} {\hat{x}_1} \right) + \frac {2\hat{x}_2}{|S|} \sin \arctan \left( \frac {\hat{x}_2} {\hat{x}_1} \right) \right) = \frac {-1} {|S|} \sqrt{\hat{x}_1^2 + \hat{x}_2^2}
	\end{align*}		
	for $\lambda_1 \neq 0$ and $u_1^\star = \pm \hat{x}_2 / |S|$ for $\lambda_1 = 0$ and $\hat{x}_1 = 0$. %
	Hence, in both cases we obtain the objective value
	\begin{equation}
		\int_S u_1^\star(t)^2\, \mathrm{d}t = \frac {\hat{x}_1^2 + \hat{x}_2^2}{|S|}, \nonumber 
	\end{equation}
	which is minimal for $|S| = T$. Hence, $S = [0,T]$ holds for the optimal control, which shows that the optimal value of~\eqref{SimplifiedOCPu1u3} is $(\hat{x}_1^2 + \hat{x}_2^2)/T$.
	
	Next, we consider~\eqref{SimplifiedOCPu2}, which is a linar quadratic OCP with zero-terminal constraint. %
	Here, the Hamiltonian is $\mathcal{H}(x_3,\lambda_3,u_2) = \lambda_0 u_2^2 + \lambda_3 u_2$. %
	Again, the differentiation of~$\mathcal{H}$ w.r.t.\ the state $x_3$ yields that the adjoint~$\lambda_3$ is constant. %
	Moreover, the abnormal multiplier can be set to one (otherwise $\mathcal{H}_{u_2} = 0$ imposes also $\lambda_3 = 0$ --- a contradiction). %
	Hence, we get $\lambda_3 = -2u_2$ from the necessary optimality condition $\mathcal{H}_{u_2} = 0$. Then, the terminal constraint implies $u^\star_2(t) = - \hat{x}_3/T$ and, thus, $\int_0^T u_2^\star(t)^2\,\mathrm{d}t = \hat{x}^2_3/T$.
	
	Adding up the two computed optimal values shows the assertion.
\end{proof}

\section*{Acknowledgements}
K.~Fla{\ss}kamp thanks L.~L\"uttgens and S.~Roy for helpful discussions on the mobile robot example, in particular for the derivation of the Lie algebra representation used to derive the trim primitives. K.~Worthmann thanks F.~Ru\ss{}wurm for helpful discussions on the mobile robot example.

\bibliographystyle{abbrv}  

\bibliography{References}   % name your BibTeX data base

\begin{thebibliography}{10}

\bibitem{Asto96}
A.~Astolfi.
\newblock Discontinuous control of nonholonomic systems.
\newblock {\em Systems \& control letters}, 27(1):37--45, 1996.

\bibitem{Bloch}
A.~M. Bloch.
\newblock {\em Nonholonomic mechanics and control}.
\newblock Springer, 2003.

\bibitem{BuLe04}
F.~Bullo and A.~D. Lewis.
\newblock {\em Geometric Control of Mechanical Systems}, volume~49 of {\em
  Texts in Applied Mathematics}.
\newblock Springer, 2004.

\bibitem{knauer}
C.~B\"uskens and M.~Knauer.
\newblock From {WORHP} to {TransWORHP}.
\newblock In {\em Proceedings of the 5th International Conference on
  Astrodynamics Tools and Techniques}, May 2012.

\bibitem{buskens2012esa}
C.~B{\"u}skens and D.~Wassel.
\newblock The {ESA} {NLP} solver {WORHP}.
\newblock In {\em Modeling and optimization in space engineering}, pages
  85--110. Springer, 2012.

\bibitem{CairKolm18}
S.~Di~Cairano and I.~V. Kolmanovsky.
\newblock Real-time optimization and model predictive control for aerospace and
  automotive applications.
\newblock In {\em 2018 Annual American Control Conference (ACC)}, pages
  2392--2409, 2018.

\bibitem{Dubins1957}
L.~E. Dubins.
\newblock On curves of minimal length with a constraint on average curvature,
  and with prescribed initial and terminal positions and tangents.
\newblock {\em American Journal of Mathematics}, 79(3):497--516, 1957.

\bibitem{FaulFlas19}
T.~Faulwasser, K.~Fla\ss{}kamp, S.~Ober-Bl\"{o}baum, and K.~Worthmann.
\newblock Towards velocity turnpikes in optimal control of mechanical systems.
\newblock In {\em Proc. 11th IFAC Symp. Nonlinear Control Systems (NOLCOS)},
  2019.

\bibitem{FHO15}
K.~Fla{\ss}kamp, S.~{Hage-Packh\"auser}, and S.~{Ober-Bl\"obaum}.
\newblock Symmetry exploiting control of hybrid mechanical systems.
\newblock {\em Journal of Computational Dynamics}, 2(1):25--50, 2015.

\bibitem{FOK12}
K.~Fla{\ss}kamp, S.~{Ober-Bl\"obaum}, and M.~Kobilarov.
\newblock Solving optimal control problems by exploiting inherent dynamical
  systems structures.
\newblock {\em Journal of Nonlinear Science}, 22(4):599--629, 2012.

\bibitem{Font01}
F.~A. Fontes.
\newblock A general framework to design stabilizing nonlinear model predictive
  controllers.
\newblock {\em Systems \& Control Letters}, 42(2):127--143, 2001.

\bibitem{Frazzoli2001}
E.~Frazzoli.
\newblock {\em Robust Hybrid Control for Autonomous Vehicle Motion Planning}.
\newblock PhD thesis, Massachusetts Institute of Technology, 2001.

\bibitem{FrazBull02}
E.~Frazzoli and F.~Bullo.
\newblock On quantization and optimal control of dynamical systems with
  symmetries.
\newblock In {\em Proceedings of the 41st IEEE Conference on Decision and
  Control}, volume~1, pages 817--823, 2002.

\bibitem{FrDaFe05}
E.~Frazzoli, M.~Dahleh, and E.~Feron.
\newblock Maneuver-based motion planning for nonlinear systems with symmetries.
\newblock {\em IEEE Transactions on Robotics}, 21(6):1077--1091, 2005.

\bibitem{GariKobi15}
G.~Garimella and M.~Kobilarov.
\newblock Towards model-predictive control for aerial pick-and-place.
\newblock In {\em 2015 IEEE international conference on robotics and automation
  (ICRA)}, pages 4692--4697, 2015.

\bibitem{GiseDoan13}
P.~Giselsson, M.~D. Doan, T.~Keviczky, B.~De~Schutter, and A.~Rantzer.
\newblock Accelerated gradient methods and dual decomposition in distributed
  model predictive control.
\newblock {\em Automatica}, 49(3):829--833, 2013.

\bibitem{GrunePannek2017}
L.~Gr{\"{u}}ne and J.~Pannek.
\newblock {\em {Nonlinear Model Predictive Control: Theory and Algorithms}}.
\newblock Communications and Control Engineering. Springer, London, 2017.

\bibitem{GrunPann10}
L.~Gr{\"{u}}ne, J.~Pannek, M.~Seehafer, and K.~Worthmann.
\newblock {Analysis of unconstrained nonlinear MPC schemes with varying control
  horizon}.
\newblock {\em SIAM Journal on Control and Optimization}, 48(8):4938--4962,
  2010.

\bibitem{GrueneWorthmann2010DMPC}
L.~Gr\"une and K.~Worthmann.
\newblock A distributed {NMPC} scheme without stabilizing terminal constraints.
\newblock In {\em {D}istributed {D}ecision {M}aking and {C}ontrol}. Springer,
  2012.

\bibitem{GuHu05}
D.~Gu and H.~Hu.
\newblock A stabilizing receding horizon regulator for nonholonomic mobile
  robots.
\newblock {\em IEEE Transactions on Robotics}, 21(5):1022--1028, 2005.

\bibitem{GuptKala15}
R.~Gupta, U.~V. Kalabi{\'c}, S.~Di~Cairano, A.~M. Bloch, and I.~V. Kolmanovsky.
\newblock Constrained spacecraft attitude control on so (3) using fast
  nonlinear model predictive control.
\newblock In {\em Proc. IEEE 2015 American Control Conf. (ACC)}, pages
  2980--2986, 2015.

\bibitem{Hous2011Automatica}
B.~Houska, H.~J. Ferreau, and M.~Diehl.
\newblock {An auto-generated real-time iteration algorithm for nonlinear {MPC}
  in the microsecond range}.
\newblock {\em Automatica}, 47(10):2279--2285, 2011.

\bibitem{Sussmann}
H.~J.~Sussmann.
\newblock Symmetries and integrals of motion in optimal control.
\newblock {\em Banach Center Publications}, 32, 11 1996.

\bibitem{JereGoul14}
J.~L. Jerez, P.~J. Goulart, S.~Richter, G.~A. Constantinides, E.~C. Kerrigan,
  and M.~Morari.
\newblock Embedded online optimization for model predictive control at
  megahertz rates.
\newblock {\em IEEE Trans. Automatic Control}, 59(12):3238--3251, 2014.

\bibitem{KalaGupt17}
U.~V. Kalabi{\'c}, R.~Gupta, S.~Di~Cairano, A.~M. Bloch, and I.~V. Kolmanovsky.
\newblock {MPC} on manifolds with an application to the control of spacecraft
  attitude on {SO} (3).
\newblock {\em Automatica}, 76:293--300, 2017.

\bibitem{KaraFraz11}
S.~Karaman and E.~Frazzoli.
\newblock Sampling-based algorithms for optimal motion planning.
\newblock {\em The international journal of robotics research}, 30(7):846--894,
  2011.

\bibitem{Keerthi1988}
S.~Keerthi and E.~Gilbert.
\newblock Optimal infinite horizon feedback laws for a general class of
  constrained discrete-time systems: stability and moving horizon
  approximations.
\newblock {\em J. Optim. Theory Appl.}, 57:265--293, 1988.

\bibitem{Ko08}
M.~Kobilarov.
\newblock {\em Discrete geometric motion control of autonomous vehicles}.
\newblock PhD thesis, University of Southern California, USA, 2008.

\bibitem{Lav06}
S.~M. LaValle.
\newblock {\em Planning Algorithms}.
\newblock Cambridge University Press, 2006.

\bibitem{lee1967foundations}
E.~B. Lee and L.~Markus.
\newblock Foundations of optimal control theory.
\newblock Technical report, Minnesota Univ Minneapolis Center for Control
  Sciences, 1967.

\bibitem{MarsdenRatiu}
J.~E. Marsden and T.~S. Ratiu.
\newblock {\em Introduction to mechanics and symmetry}, volume~17 of {\em Texts
  in Applied Mathematics}.
\newblock Springer, 2nd edition, 1999.

\bibitem{MullWort17}
M.~A. M{\"{u}}ller and K.~Worthmann.
\newblock {Quadratic costs do not always work in MPC}.
\newblock {\em Automatica}, 82:269--277, 2017.

\bibitem{Murray1994}
R.~M. Murray, S.~S. Sastry, and L.~Zexiang.
\newblock {\em A Mathematical Introduction to Robotic Manipulation}.
\newblock CRC Press, Inc., Boca Raton, FL, USA, 1st edition, 1994.

\bibitem{OP18}
S.~{Ober-Bl\"obaum} and S.~Peitz.
\newblock Explicit multiobjective model predictive control for nonlinear
  systems with symmetries.
\newblock Submitted, ar{X}iv:1809.06238.

\bibitem{Frazzoli2016}
B.~Paden, M.~\v{C}\'{a}p, S.~Z. Yong, D.~Yershov, and E.~Frazzoli.
\newblock A survey of motion planning and control techniques for self-driving
  urban vehicles.
\newblock {\em IEEE Transactions on Intelligent Vehicles}, 1(1):33--55, 2005.

\bibitem{ParoLaug96}
I.~Paromtchik and C.~Laugier.
\newblock Autonomous parallel parking of a nonholonomic vehicle.
\newblock In {\em Proceedings of the IEEE Intelligent Vehicles Symposium},
  pages 13--18, 1996.

\bibitem{PSOB+17}
S.~Peitz, K.~Sch{\"{a}}fer, S.~Ober-Bl{\"{o}}baum, J.~Eckstein,
  U.~K{\"{o}}hler, and M.~Dellnitz.
\newblock {A Multiobjective MPC Approach for Autonomously Driven Electric
  Vehicles}.
\newblock {\em IFAC PapersOnLine}, 50(1):8674--8679, 2017.

\bibitem{RawlingsMayneDiehl2017}
J.~B. Rawlings, D.~Q. Mayne, and M.~M. Diehl.
\newblock {\em Model Predictive Control: Theory, Computation, and Design}.
\newblock Nob Hill Publishing, 2017.

\bibitem{Reble_Allgower_2012}
M.~Reble and F.~Allg{\"o}wer.
\newblock Unconstrained model predictive control and suboptimality estimates
  for nonlinear continuous-time systems.
\newblock {\em Automatica}, 48(8):1812--1817, 2012.

\bibitem{ReedsShepp1990}
J.~A. Reeds and L.~A. Shepp.
\newblock Optimal paths for a car that goes both forwards and backwards.
\newblock {\em Pacific Journal of Mathematics}, 145(2):367--393, 1990.

\bibitem{SchuWort18}
M.~{Schulze Darup} and K.~Worthmann.
\newblock Tailored {MPC} for mobile robots with very short prediction horizons.
\newblock In {\em Proceedings of the 2018 European Control Conference (ECC
  2018), Limassol, Cyprus}, pages 1361--1366, 2018.

\bibitem{Sontag1998}
E.~Sontag.
\newblock {\em {Mathematical Control Theory - Deterministic Finite Dimensional
  Systems}}.
\newblock Number~6 in Texts in Applied Mathematics. Springer-Verlag New York,
  second edition, 1998.

\bibitem{TunaMessinaTeel2006}
S.~E. Tuna, M.~J. Messina, and A.~R. Teel.
\newblock {S}horter horizons for model predictive control.
\newblock In {\em Proc. Amer. Control Conf.}, Minneapolis, MN, USA, 2006.

\bibitem{Wort2015NMPC}
K.~Worthmann, M.~W. Mehrez, M.~Zanon, G.~K.~I. Mann, R.~G. Gosine, and
  M.~Diehl.
\newblock {Regulation of Differential Drive Robots using Continuous Time MPC
  without Stabilizing Constraints or Costs}.
\newblock {\em IFAC-PapersOnLine}, 48(23):129--135, 2015.

\bibitem{Wort2016CST}
K.~Worthmann, M.~W. Mehrez, M.~Zanon, G.~K.~I. Mann, R.~G. Gosine, and
  M.~Diehl.
\newblock {Model Predictive Control of Nonholonomic Mobile Robots Without
  Stabilizing Constraints and Costs}.
\newblock {\em IEEE Transactions on Control Systems Technology},
  24(4):1394--1406, 2016.

\bibitem{Wort2014SICON}
K.~Worthmann, M.~Reble, L.~Gr{\"{u}}ne, and F.~Allg{\"{o}}wer.
\newblock {The Role of Sampling for Stability and Performance in Unconstrained
  Nonlinear Model Predictive Control}.
\newblock {\em SIAM Journal on Control and Optimization}, 52(1):581--605, 2014.

\bibitem{ZeilJone11}
M.~N. Zeilinger, C.~N. Jones, and M.~Morari.
\newblock Real-time suboptimal model predictive control using a combination of
  explicit {MPC} and online optimization.
\newblock {\em IEEE Trans. Automatic Control}, 56(7):1524--1534, 2011.

\end{thebibliography}
\end{document}